\documentclass[11pt]{article}

\usepackage{times}
\usepackage{amsthm}
\usepackage{amsmath}
\usepackage{amssymb}
\usepackage{mathrsfs}
\usepackage{pb-diagram}
\usepackage{color}

\usepackage{hyperref}

\renewcommand\eqref[1]{(\ref{#1})} 

\def\A{{\mathcal A}}

\def\th{\theta}
\def\Th{\Theta}
\def\N{\mathbb{N}}
\def\Z{\mathbb{Z}}

\def\H{\mathcal H}

\def\R{\mathbb{R}}

\def\({\left(}
\def\[{\left[}
\def\){\right)}
\def\]{\right]}

\def\si{\sigma}
\def\Si{\Sigma}

\def\p{\parallel}
\def\<{\langle}
\def\>{\rangle}

\def\d{\mathrm{d}}
\def\r{\mathrm{r}}

\def\Cl{\mathrm{Cl}}

\def\Ouv{\mathcal O}
\def\T{\mathbf T}

\def\U{\mathcal U}

\newtheorem{Theorem}{Theorem}[section]
\newtheorem{Remark}[Theorem]{Remark}
\newtheorem{Lemma}[Theorem]{Lemma}

\newtheorem{Corollary}[Theorem]{Corollary}
\newtheorem{Proposition}[Theorem]{Proposition}
\newtheorem{Definition}[Theorem]{Definition}
\newtheorem{Example}[Theorem]{Example}

\newtheorem{Setting}[Theorem]{Setting}

\numberwithin{equation}{section}

\begin{document}


\title{On Persson's Formula; an \'Etale Groupoid Approach}

\date{\today}

\author{M. M\u antoiu
}
\date{\small}
\maketitle \vspace{-1cm}


\begin{abstract}
Persson's formula expresses the infimum of the essential spectrum of a suitable self-adjoint Schr\"odinger operator in $\R^n$ in terms of the lower spectral points of a family of restrictions of the operator to complements of relatively compact subsets. It has been extended to other situations. We present an approach based on $C^*$-algebras associated to \'etale groupoids. In such a setting there are intrinsic versions, referring to self-adjoint elements of the groupoid $C^*$-algebras. When representations in Hilbert spaces are considered, the results not always involve only the essential spectrum. The range of applications is quite distinct from the traditional one. We indicate examples related to symbolic dynamics and band dominated operators on discrete metric spaces. The treatment needs only a small amout of symmetry. Even when group actions are involved, restrictions to non-invariant subsets are needed and have to be treated carefully.
\end{abstract}


\section{Introduction}\label{oameni}

Let us describe Persson's result \cite{Pe} in a simplified form,. In the Hilbert space $\H:=L^2(\R^n)$ one considers the positive Laplace operator $\Delta:=-\sum_{j=1}^n\!\partial_j^2$\,, the multiplication operator $V$ by a function $V:\R^n\to\R$\,, and the Schr\"odinger operator $H:=\Delta+V$. We also denote by $\mathcal K(\R^n)$ the family of all compact subsets of $\R^n$, and  if $A$ is an open subset of $\R^n$, we write $\mathcal D(A)$ for $C^\infty_{\rm c}(A)$\,. For any subset $B$ of $\R^n$, the notation $B^{\rm c}$ stands for its complement. For simplicity, and since the technical conditions on the potential are not very important here, we require much more than in \cite{Pe}. 

\begin{Theorem}\label{persson}
If $\,V$ is real, continuous and bounded, then one has 
\begin{equation}\label{bathica}
\inf{\sf sp}_{\rm ess}(H)=\underset{K\in\mathcal K(\R^n)}{\sup}\,\underset{\underset{\p u\p=1}{u\in\mathcal D(K^{\rm c})}}{\inf}\<Hu,u\>\,.
\end{equation}
\end{Theorem}

The expression $\underset{\underset{\p u\p=1}{u\in\mathcal D(K^{\rm c})}}{\inf}\<Hu,u\>$ is actually the infimum of the Dirichlet operator 
$H_{K^{\rm c}}\!:=\Delta_{K^{\rm c}}+V$ in $\H_{K^{\rm c}}\!:=L^2\big(K^{\rm c}\big)$\,. Thus \eqref{bathica} can be rewritten as
\begin{equation}\label{badica}
\inf{\sf sp}_{\rm ess}(H)=\underset{K\in\mathcal K(\R^n)}{\sup}\inf{\sf sp}\big(H_{K^{\rm c}}\big)\,.
\end{equation}
One finds in \cite{Ag,CFKS,DLMSY,Gr,KL,LS} (see also the references therein) many other results of this type. Among others, one covers Laplace-Beltrami operators on Riemannian manifolds, operators on graphs and $\alpha$-stable processes in the Euclidean space. We do not try to give a comprehensive overview of these references, since our interest and techniques point in a quite different direction (but we refer to \cite{LS} for a recent very useful presentation). Roughly, one can say that we are confined to bounded operators associated to a discrete setting, exploiting the structure of certain \'etale groupoid $C^*$-algebras. Extensions to other situations seem possible, but they would need extra effort and some new ideas; hopefully this will be be the subject of a subsequent article.

\smallskip
We start by describing one of our results. Let us fix an infinite uniformly discrete metric space $(X_0,\delta)$ with bounded geometry, satisfying Yu's Property A (the technical terms will be explained in Section \ref{grafologi}).  We consider linear bounded operators ${\sf H}$ in the Hilbert space $\ell^2(X_0)$ given in matrix-form by
\begin{equation}\label{frayeur}
[{\sf H}(u)](x)=\sum_{y\in X_0}\!H(x,y)u(y)\,,\quad\forall\,x\in X_0\,.
\end{equation}
A simple class is that of {\it band operators}, for which $\sup_{x,y}|H(x,y)|<\infty$ and there exists $r>0$ such that $H(x,y)=0$ if $\delta(x,y)>r$.
A linear bounded operator is {\it band dominated} if it is a norm-operator limit of band operators.
Let us define $\mathbf S_0$ to be the family of subsets $M_0\subset X_0$ with finite complement.
Assuming that \eqref{frayeur} is a band dominated operator in $\ell^2(X_0)$\,, then
\begin{equation}\label{freier}
\big[{\sf H}_{M_0}(v)\big](a)=\sum_{b\in M_0}\!H(a,b)v(b)\,,\quad\forall\,a\in M_0
\end{equation}
defines a band dominated operator in $\ell^2(M_0)$ for every $M_0\in\mathbf S_0$\,.
For every linear bounded operator ${\sf T}$, we denote by ${\sf sp}({\sf T})$ its spectrum and by ${\sf sp}_{\rm ess}({\sf T})$ its essential spectrum.

\begin{Theorem}\label{uvert}
If $\,{\sf H}$ is self-adjoint, then
\begin{equation}\label{cealalta}
\inf{\sf sp_{\rm ess}}({\sf H})=\sup_{M_0\in\mathbf S_0}\!\big[\inf {\sf sp}\big({\sf H}_{M_0}\big)\big]\quad{\rm and}\quad\sup{\sf sp_{\rm ess}}({\sf H})=\inf_{M\in\mathbf S_0}\!\big[\sup {\sf sp}\big({\sf H}_{M_0}\big)\big]\,.
\end{equation}
\end{Theorem}

To get this result, for us it is important that the band dominated operators form a unital $C^*$-subalgebra of $\mathbb B\big[\ell^2(X_0)\big]$, called {\it the uniform Roe algebra of the metric space}. This is also true if one replaces $X_0$ by its subset $M_0$\,. The operator ${\sf H}_{M_0}$ is a compression of ${\sf H}$ to the Hilbert subspace $\ell^2(M_0)\subset\ell^2(X_0)$\,, but the operation ${\sf H}\to{\sf H}_{M_0}$ is not implemented by a morphism of $C^*$-algebras. In a symmetric framework, if $X_0$ happens to be a (discrete) group (and Roe's algebra is also a crossed product defined by the action on its Stone-\u Cech compactification), $M_0$ is no longer a group and the symmetry is greatly reduced. 

\smallskip
Two ideas emerge: (a) if $C^*$-algebras are to be used in proving Persson's type formulae, one has to face situations when some of these $C^*$-algebras are connected in a non-multiplicative way and (b) the relationship between the configuration spaces on which the operators act and the $C^*$-algebras that naturally contain them should involve a mathematical object allowing very weak forms of symmetry.
A convenient solution is to work with {\it locally compact groupoids}, that have interesting and well-studied associated $C^*$-algebras \cite{Pa,Re,Wi1} as soon as they admit Haar systems. They are very general concepts, allowing the unified treatment of many particular cases.

\smallskip
For simplicity, and because of the applications we had in view, some technical restrictions have been adopted. One was {\it amenability} (for band dominated operators this is related to Yu's Property A). More restrictive is our assumption that the groupoid is {\it \'etale}, leading roughly to the discreteness of the models that we can treat and to boundedness of the relevant operators. Since they are bounded, the supremum of the (essential) spectrum can be equally treated. Non-\'etale groupoids would surely be useful, but some of the arguments (as the non-invariant restriction issue, to be discussed below) would be more difficult. The "continuous models" that they are supposed to cover also present an extra non-trivial difficulty: the natural operators are unbounded, so the connection with $C^*$-algebras should be done via some affiliation techniques, in terms of functional calculus and resovent families. But making this compatible with the groupoid and $C^*$-algebraic formalism seems to be quite delicate and deserves further investigation.

\smallskip
Section \ref{siena} briefly describes the type of locally compact groupoids that we consider and their $C^*$-algebras. We think that this is not the right place to define topological groupoids. But a reader familiar with the very basic facts will be able to follow. Restrictions to closed but non-invariant sets of units of the groupoid will be needed, and this is discussed here. Since non-invariant restrictions might not inherit a continuous Haar system from the the initial groupoid, we introduce {\it tame subsets}. In our \'etale setting there are many natural clopen subsets, that are tame.

\smallskip
Part of our proof of the Persson type formula rely on continuity arguments. The basic topological space used to express this is the family of all closed sets of the compact unit space of the groupoid. In Section \ref{troscanit} we recall briefly its Fell (hyperspace) topology. This allows stating an intrinsic spectral result referring to self-adjoint elements of the groupoid $C^*$-algebra. The proof is given in Section \ref{stendhal}.

\smallskip
Continuity issues are nicely treated if one has a continuous field of groupoids (leading to continuous field of $C^*$-algebras). In Section \ref{danganit} we make the necessary preparations, inspired by the approach in \cite{LR,BBdN1,Bec}; see also \cite{AMP,BBdN2,Be}. Note that in \cite{BBdN1,Bec} the important {\it tautological groupoid} is constructed only with restrictions to invariant closed sets of units, but this was enough to get results to which we have no access here. On the other hand, in the present article, considering families of {\it non-invariant subsets} is critical. Such situation needs some extra arguments and some suitable extra assumptions. The reference \cite{BBdN1} covers groupoid $2$-cocycles, leading to {\it twisted} groupoid $C^*$-algebras, that can be used to model magnetic fields. We could do the same, but for simplicity we treat only the non-twisted case.

\smallskip
To arrive at self-adjoint operators (Hamiltonians), one uses Hilbert space representations of the groupoid $C^*$-algebra. For such algebras there are proliferations of representations, but those we are going to use (regular representations induced from units) are natural. Section \ref{balanganit} put them into evidence and deduce an operatorial Persson type formula from the previous intrinsic Theorem \ref{dorita}. There is still no essential spectra in view.

\smallskip
Section \ref{talanganit} is dedicated to converting all these into a result on the essential spectrum. For this a special type of groupoids is required, which do appear naturally in applications. Their unit space must possess an open dense orbit $X_0$ with trivial isotropy groups. This allows to introduce a canonical {\it vector representation} in $\ell^2(X_0)$ and to identify the $C^*$-algebra of the (invariant) groupoid restriction to $X_0$ with the ideal of all the compact operators in this Hilbert space. In its turn this leads to identify a certain spectrum appearing in Theorem \ref{dorita} with the essential spectrum of the relevant operator acting in $\ell^2(X_0)$\,. Theorem \ref{consecinte}, although not the most general, may be considered our main abstract result.

\smallskip
In Section \ref{geografi} we apply the abstract formulae to a special type of dynamical systems, belonging to the theory of symbolic dynamics: the {\it bi-sided subshifts} constructed on a finite alphabet. Very roughly, the covered operators are compressions (Toeplitz versions) of  parameter-dependent pseudodifferential operators on the group $\Z$\,. More general actions of discrete groups are also possible. The high symmetry of the initial system is compatible with the use of crossed product $C^*$-algebras. But as soon as restrictions are taken into account this is no longer possible, only actions by partial homeomorphism being still available. Groupoids tackle this conveniently. One-sided shifts are more complicated, they could also be treated by our method, but we decided to include them in a larger perspective in a future publication, including inverse semigroups and also involving other types of spectral results.

\smallskip
Section \ref{grafologi} is dedicated to the proof of Theorem \ref{uvert}. A graph-theoretical interpretation is provided.

\smallskip
A Decomposition Principle \cite{LS,Ma} says that the essential spectrum of a Hamiltonian $H$ acting in an $L^2(X)$-space coincides with the essential spectrum of its "restriction" $H_Y$ to any complement $Y$ of a relatively compact subset. Besides its intrinsic interest, this could also be a step towards proving a Persson formula. In \cite{LS} (see also references cited therein), by analytical and probabilistic methods, such a Decomposition Principle has been obtained for operators associated to regular Dirichlet forms, under an extra relative compactness assumption. In \cite{Ma} it has been proven in a different context, by twisted groupoid $C^*$-algebra methods. In the present article, such a Decomposition Principle is not needed.

\section{Manageable groupoids}\label{siena}

We begin by introducing the class of groupoids we will work with.  At certain stages less conditions would be required, but we prefer to stay in a unified framework, leaving to the reader the task to extend the generality whenever this is possible.

\begin{Definition}\label{manageable}
Let $\Xi$ be a Hausdorff locally compact $\si$-compact groupoid with compact unit space $X$, structural maps $\r, \d:\Xi\to X$  and family of composable elements 
$$
\Xi^{(2)}\!:=\{(\xi,\eta)\in\Xi\times\Xi \mid \d(\xi)=\r(\eta)\}\,.
$$ 
We say that it is {\rm manageable} if it is \'etale and amenable.
\end{Definition}

In a couple of remarks, we are going to provide some explanations or indicate some consequences.

\begin{Remark}\label{ital}
{\rm {\it The \'etale condition} means that $\r, \d$ are local homeomorphisms when seen as maps $\Xi\to\Xi$\,. Let us mention consequences of the fact that $\Xi$ is \'etale: 
\begin{itemize}
\item 
$X\subset\Xi$ is a clopen subset, 
\item
$\d,\r$ and the multiplication are continuous open maps,
\item
The fibers $\Xi_x\!:=\{\xi\in\Xi\!\mid\!\d(\xi)=x\}$ and $\Xi^x\!:=\{\xi\in\Xi\!\mid\!\r(\xi)=x\}$ are discrete in the induced topology,
\item
The counting measures on each $\Xi_x$ form a right Haar system $\{\lambda_x\!\mid\! x\in X\}$\,.
\end{itemize} 
}
\end{Remark}

\begin{Remark}\label{alghebrele}
{\rm The locally compact \'etale groupoid is called {\it amenable} \cite[Def.\,5.6.13]{BO} if there exists a net $\{\phi_i\!\mid i\in I\}$ of positive continuous compactly supported functions on $\Xi$ such that
$$
\sum_{\eta\in\Xi_{\r(\xi)}}\!\phi_i(\eta)\underset{i}{\to}1\,,\quad\sum_{\eta\in\Xi_{\r(\xi)}}\!\big\vert\phi_i(\eta\xi)-\phi_i(\eta)\big\vert\underset{i}{\to}0\,
$$
uniformly on compact subsets of $\Xi$\,.  There are reformulations of this condition and for non-\'etale groupoids the theory is more involved. The only way we make use of amenability is through its consequence outlined in the next remark.
}
\end{Remark}

\begin{Remark}\label{algebrele}
{\rm We recall that for every \'etale locally compact groupoid $\Xi$\,, the space $C_{\rm c}(\Xi)$ of all continuous compactly supported complex functions defined on $\Xi$ has a natural structure of $^*$-algebra. The product and the involution are given by
\begin{equation*}\label{roduct}
(f\star g)(\xi):=\sum_{\eta\zeta=\xi}f(\eta)g(\zeta)=\!\sum_{\zeta\in\Xi_{\d(\xi)}}\!f\big(\xi\zeta^{-1}\big)g(\zeta)\,,\quad f^\star(\xi):=\overline{f(\xi^{-1})}\,.
\end{equation*}
There are (at least) two interesting $C^*$-completions: ${\rm C}^*[\Xi]$ (the full groupoid $C^*$-algebra) and ${\rm C}^*_{\rm r}[\Xi]$ (the reduced groupoid $C^*$-algebra). In general the later is a quotient of the former, but in many situations they are isomorphic; 
this happens, in particular, if $\Xi$ is amenable, cf. \cite[Cor.\,5.6.17]{BO} for example.}
\end{Remark}

\begin{Remark}\label{aljebrele}
{\rm In our case, since $\Xi$ is \'etale and its unit space is compact, ${\rm C}^*[\Xi]={\rm C}^*_{\rm r}[\Xi]$ is unital, the unit being the characteristic function of the unit space.}
\end{Remark}

\begin{Remark}\label{han}{\rm Assuming that $\Xi$ is manageable, one has the following continuous dense embeddings
\begin{equation*}\label{siruite}
C_{\rm c}(\Xi)\hookrightarrow{\rm Hahn}(\Xi)\hookrightarrow{\rm C}^*[\Xi]\hookrightarrow C_0(\Xi)\,.
\end{equation*}
Here ${\rm Hahn}(\Xi)$ is the Banach space associated to {\it the Hahn norm}
\begin{equation*}\label{ghan}
\p\!f\!\p_{{\rm Hahn}(\Xi)}\,:=\max\Big\{\sup_{x\in X}\sum_{\xi\in\Xi_x}|f(\xi)|\,,\sup_{x\in X}\sum_{\xi\in\Xi^x}|f(\xi)|\Big\}\,.
\end{equation*}
It is also a $^*$-subalgebra of ${\rm C}^*[\Xi]$\,. The last embedding (specific for the reduced algebra and for the \'etale case) involves the $C^*$-algebra of all complex continuous functions on $\Xi$ that decay at infinity, with the uniform norm $\p\!g\!\p_\infty:=\sup_{\xi\in\Xi}|g(\xi)|$\,. One has
\begin{equation*}\label{dominari}
\p\!g\!\p_\infty\,\le\,\p\!g\!\p_{{\rm C}^*[\Xi]}\,\le\,\p\!g\!\p_{{\rm Hahn}(\Xi)}.
\end{equation*}
}
\end{Remark}

For any $A,B\subset X$ we set 
\begin{equation*}\label{mixta}
\Xi_B^A:=(\r,\d)^{-1}(A,B)=\{\xi\in\Xi\!\mid\!\r(\xi)\in A\,,\d(\xi)\in B\}\,.
\end{equation*}
If $A=B=\{x\}$ for some unit $x$\,, we use the special notation 
$$
\Xi_x^x=\{\xi\in\Xi\!\mid\!\r(\xi)=x=\d(\xi)\}
$$
for {\it the isotropy group at $x$}. Recall that an equivalence relation on the unit space is given by
\begin{equation*}\label{ghivalens}
x\cong y\ \iff\ \exists\,\xi\in\Xi\ \,{\rm such\ that}\ \,d(\xi)=x\,,\;\r(\xi)=y\,.
\end{equation*}
Then $\Xi$ is partitioned in orbits corresponding to this equivalence relation. The subset $M\subset X$ is called {\it invariant} if it is a union of orbits; equivalently, if $\d(\xi)\in M$ implies $\r(\xi)\in M$.
Plenty of non-invariant closed sets of units will be needed below, and this requires a careful treatment.

\smallskip
If $M\subset X$ is closed, (thus compact) {\rm the restricted subgroupoid} $\,\Xi(M)\!:=\Xi_M^M$ is a locally compact groupoid over the unit space $M$. It is also known to be amenable, cf. \cite[Prop.3.7]{Re}. Unfortunately, $\Xi(M)$ is not automatically \'etale (although this does hold whenever $M$ is invariant). It is always $\r$-discrete, meaning by definition that its unit space $M$ is clopen in $\Xi(M)$ (and implying discreteness of $\Xi(M)_x$ and $\Xi(M)^x$). However, to be \'etale, what is missing is precisely $\r$ restricted to $\Xi(M)$ being open, and this is not always true. In addition, when this fails, there is no Haar system; this is an important issue for us. The next result is well-known:

\begin{Lemma}\label{teim}
For the closed set $M\subset X$, the following conditions are equivalent: 
\begin{enumerate}
\item[(i)] the restriction $\d_M\!:\Xi(M)\to\Xi(M)$ is open, 
\item[(ii)] the restriction $\r_M\!:\Xi(M)\to\Xi(M)$ is open, 
\item[(iii)] the multiplication $\Xi(M)^{(2)}\!\to\Xi(M)$ is open, 
\item[(iv)] the groupoid $\Xi(M)$ is \'etale, 
\item[(v)] restricting to each $\Xi(M)_x\!=\Xi^M_x$ the counting measure, one gets a Haar system $\big\{\lambda^M_x\,\big\vert\, x\in M\big\}$ on $\Xi(M)$\,.
\end{enumerate}
\end{Lemma}

\begin{Definition}\label{treim}
The closed set $M\subset X$ is called {\rm tame} if it satisfies the equivalent conditions of Lemma \ref{teim}. The family of all closed (thus compact) subsets of $X$ will be denoted by $\Cl(X)$\,; the tame subsets form a subfamily $\Cl_{\rm tm}(X)$\,.
\end{Definition}

\begin{Lemma}\label{daca}
If $M$ is invariant or clopen, it is tame.
\end{Lemma}

\begin{proof}
The clopen case is simple. If $M$ is open, then $\Xi(M)$ is also open, so its open subsets are precisely those which are open in $\Xi$ and contained in $\Xi(M)$\,. Then clearly $\d_M\!:\Xi(M)\to\Xi(M)$ is an open map.

\smallskip
We treat the invariant case. If one shows that 
$$
\d\big(A\cap\Xi(M)\big)=\d(A)\cap\Xi(M)\,,\quad\forall\,A\subset\Xi\ {\rm open},
$$ 
then indeed $d_M\equiv\d\big\vert_{\Xi(M)}\!:\Xi(M)\to\Xi(M)$ is open. But
$$
x\in\d\big(A\cap\Xi(M)\big)\ \Leftrightarrow\ \exists\,\xi\in A\ \,{\rm with}\ \,x=\d(\xi)\in M,\ \r(\xi)\in M,
$$
while
$$
x\in\d(A)\cap\Xi(M)\ \Leftrightarrow\ x\in\d(A)\cap M\ \Leftrightarrow\ \exists\,\xi\in A\ \,{\rm with}\ \,x=\d(\xi)\in M.
$$
A priori $\d\big(A\cap\Xi(M)\big)\subset\d(A)\cap\Xi(M)$\,. The two conditions are identical in the invariant case, in which $\d(\xi)\in M\,\Rightarrow\,\r(\xi)\in M$.
\end{proof}

\begin{Remark}\label{onrestrictions}
{\rm We go on assuming that $\Xi$ is manageable. Some comments on restrictions are needed. Generally, they are first defined at the $C_{\rm c}$-level and then are extended at the level of $C^*$-algebras. So, if $M\subset X$ is closed, one sets 
\begin{equation}\label{zdrictiile}
\rho_M:C_{\rm c}(\Xi)\to C_{\rm c}\big(\Xi(M)\big),\quad \rho_M(f):=f|_{\Xi(M)}\,.
\end{equation}
Assume first that $M$ is invariant. Then $\rho_M$ is a $^*$-morphism and it is known that it extends to a $C^*$-morphism
\begin{equation*}\label{sdrictiile}
\rho_M:{\rm C}^*[\Xi]\to{\rm C}^*\big[\Xi(M)\big].
\end{equation*}
{\it If $M$ is not invariant \eqref{zdrictiile} is still linear and involutive, but it is no longer a morphism.} Note the formula
\begin{equation}\label{banati}
\rho_M(f\star g)-\rho_M(f)\star_M\rho_M(g)=\varsigma_M(f,g)\,,
\end{equation}
where for every $\xi\in\Xi(M)$
\begin{equation*}\label{karenc}
\big[\varsigma_M(f,g)\big](\xi):=\!\!\underset{\eta\zeta=\xi,\,\d(\eta)\notin M}{\sum}\!f(\eta)g(\zeta)\,.
\end{equation*}
(If $M$ is invariant, then $\r(\eta)=\r(\xi)\in M$ implies $\d(\eta)\in M$ and thus $\varsigma_M(f,g)=0$\,.)
However, it is shown in \cite[Lemma 5.6.(i)]{Ma} that \eqref{zdrictiile} extends to a contraction $\rho_M:{\rm C}^*[\Xi]\to{\rm C}^*\big[\Xi(M)\big]$. It is even simpler to check that $\rho_M:{\rm Hahn}(\Xi)\to{\rm Hahn}\big[\Xi(M)\big]$ is a well-defined contraction. All these are true in general, but in our manageable case, as mentioned before, elements of $C^*[\Xi]$ may be seen as continuous functions $f:\Xi\to\mathbb C$ that decay at infinity, and $\rho_M$ has the usual sense of restriction of functions on the entire groupoid $C^*$-algebra.}
\end{Remark}

\section{Framework and statement of the intrinsic result}\label{troscanit}

On the set $\Cl(X)$ of all closed subsets of the compact space $X$ one considers {\it the Fell topology}, defined as follows. For any compact subset $K\subset X$ and for any finite family $\mathfrak O$ of open sets of $X$, we define
\begin{equation*}\label{vietoris}
\U(K;\mathfrak O):=\big\{A\in\Cl(X)\mid A\cap K =\emptyset\ \hbox{and} \ A\cap \Ouv \neq \emptyset \ \hbox{for any }\ \Ouv\in \mathfrak O\big\}\,.
\end{equation*}
The family of all such sets $\U(K;\mathfrak O)$ defines a basis for the topology on $\Cl(X)$\,. It is known that $\Cl(X)$ endowed with this topology becomes a compact (and Hausdorff) space \cite[Thm.~1]{Fe}. 

\smallskip
There is a simple description of convergence in this topology \cite[Lemma\,H.2]{Wi}: 

\begin{Lemma}\label{indiciaza}
The net $(A_i)_{i\in I}$ converges to $A$ in the Fell topology on $\Cl(X)$ if and only if 
(a) for every $A_i\ni\gamma_i\to\gamma\in X$ one has $\gamma\in A$ (outer convergence) and 
(b) if $\,\gamma\in A$\,, there is a subnet $\big(A_{i_j}\big)_{j\in J}$ and elements $\gamma_j\in A_{i_j}$ such that $\gamma_j\to\gamma$ (inner convergence).
\end{Lemma}

If $A\subset A_i$ for every $i$\,, (b) always holds. If $A\supset A_i$ for every $i$\,, (a) is automatically satisfied\,.

\begin{Setting}\label{sett} 
Let $\Xi$ be a manageable groupoid with unit space $X$. Let $\,\T$ be a Fell-closed subfamily of $\,\Cl(X)$ formed of tame subsets, containing a closed invariant subset $X_\infty$ which is an accumulation point of $\,\T$ and such that $X_\infty\subset M$ for every $M\in\T_0:=\T\!\setminus\!\{X_\infty\}$\,. 
\end{Setting}

\begin{Example}\label{sifacut}
{\rm A particular situation is to ask $\T_0$ to be a net of clopen sets Fell-converging to $X_\infty$\,, all these sets strictly containing $X_\infty$\,.}
\end{Example}

\begin{Theorem}\label{dorita}
Assume that Setting \ref{sett} holds. For every self-adjoint $h\in C^*[\Xi]$ we set 
\begin{equation*}\label{infinit}
h_\infty\!:=\rho_{X_\infty}\!(h)\in C^*\big[\Xi(X_\infty)\big]\,,
\end{equation*}
\begin{equation*}\label{infinita}
h_M\!:=\rho_M(h)\in C^*\big[\Xi(M)\big]\,,\ \forall\,M\in\T_0\,.
\end{equation*}
One has
\begin{equation}\label{frik}
\inf{\sf sp}\big(h_\infty\big)=\sup_{M\in\T_0}\!\big[\inf {\sf sp}\big(h_M\big)\big]\,,
\end{equation}
\begin{equation}\label{frig}
\sup{\sf sp}\big(h_\infty\big)=\inf_{M\in\T_0}\!\big[\sup {\sf sp}\big(h_M\big)\big]\,.
\end{equation}
\end{Theorem}

\section{Continuous fields of groupoids and $C^*$-algebras}\label{danganit}

We now introduce a continuous field of groupoids associated to the given manageable groupoid. The construction is inspired by a similar construction provided in \cite{BBdN1}. The main difference is that we consider sets of units $M$ that are not invariant.  Assuming Setting \ref{sett} ($X_\infty$ will not be particularly important here, but $\T$ has to be Fell closed), let us introduce
\begin{itemize}
\item
the set $\,\Gamma\!_\T(\Xi):=\big\{(M,\xi)\mid M\in\T \ \hbox{and}\ \xi \in \Xi(M)\big\}\subset\T\times \Xi$\,,
\item 
the space of units $\,\Gamma\!_\T(\Xi)^{(0)}\!:=\big\{(M,x)\mid M\in \T \ \hbox{and}\ x \in M\big\}\subset\T\times X$,
\item 
two maps $\r_\T\!:=({\rm id}\times\r)\vert_{\Gamma\!_\T(\Xi)}$\,, $\d_\T\!:=({\rm id}\times\d)\vert_{\Gamma\!_\T(\Xi)}$ from $\Gamma\!_\T(\Xi)$ to $\Gamma\!_\T(\Xi)^{(0)}$, actually defined by  
$$
\r_\T(M,\xi):=\big(M,\r(\xi)\big)\,,\quad \d_\T(M,\xi):=\big(M,\d(\xi)\big)\,,
$$
\item 
the set $\Gamma\!_\T(\Xi)^{(2)}$ of composable pairs given by
$$
\Gamma\!_\T(\Xi)^{(2)}\!:=\big\{\big((M,\xi),(M,\eta)\big)\mid M\in \T\,,\ \xi,\eta\in \Xi(M)\ \hbox{and}\ \r(\eta)=\d(\xi)\big\}\subset \Gamma\!_\T(\Xi)\times\Gamma\!_\T(\Xi)\,,
$$ 
with composition $(M,\xi)(M,\eta):=(M,\xi\eta)$\,,
\item 
the inverse map given by $(M,\xi)^{-1}\!:=(M,\xi^{-1})$\,.
\end{itemize}

\begin{Proposition}\label{ciocardel}
The above defined $\Gamma\!_\T(\Xi)$ is a Hausdorff locally compact $\si$-compact groupoid having a right Haar system.
\end{Proposition}

\begin{proof}
We may interpret $\T$ as a (trivial) compact groupoid, reduced to units, with the only allowed compositions $MM\!:=M$ for $M\in\T$. The product $\T\times\Xi$ is naturally a locally compact Hausdorff groupoid, and its unit space is $\T\times X$. 
We start by showing that $\Gamma\!_\T(\Xi)$ is a closed subgroupoid of this product. The topological properties are then direct consequences.

\smallskip
Since each $\Xi(M)$ is a subgroupoid of $\Xi$\,, it is trivial to check that $\Gamma\!_\T(\Xi)$ is stable under compositions and inversions. 
In fact, it may also be seen as the disjoint union groupoid 
$$
\Gamma\!_\T(\Xi)\equiv\bigsqcup_{M\in\T}\!\Xi(M)\,.
$$
Let us now show that $\Gamma\!_\T(\Xi)$ is a closed subset of $\T\times \Xi$\,. For that purpose, consider a net $(M_i,\xi_i)$ in $\Gamma\!_\T(\Xi)$ converging to $(M,\xi)$ in the product space $\T\times \Xi$\,; one has to show that $(M,\xi)\in\Gamma\!_\T(\Xi)$\,, i.\,e.\; that $\d(\xi),\r(\xi)\in M$. By continuity, $\xi_i\to\xi$ implies that $M_i\ni\d(\xi_i)\to\d(\xi)$ and $M_i\ni\r(\xi_i)\to\r(\xi)$ in $X$ and then applying the point (a) of Lemma \ref{indiciaza} finishes the proof.

\smallskip
On each $\d$-fibre $\Xi_x\!:=\{\xi\in\Xi\!\mid\!\d(\xi)=x\}$ (discrete in the relative topology) we consider the counting measure $\lambda_x$\,. By tameness of $M\in\T$, on $\Xi(M)_x\!=\Xi^M_x$ one has the counting measure denoted by $\lambda^M_x$, getting a right Haar system on the groupoid $\Xi(M)$\,. Let us denote by $\delta_M$ the Dirac measure at $M\in\T$\,, and consider the product measure $\delta_M\times \lambda^M_x$ on the fibre $\Gamma\!_\T(\Xi)_{(M,x)}=\{M\}\times\Xi(M)_x$\,. Then clearly 
$$
\big\{\delta_M\times \lambda^M_x\,\big\vert\,M \in\T\,, x \in M\big\}
$$ 
defines a right Haar system on  $\Gamma\!_\T(\Xi)$\,.
\end{proof}

\begin{Remark}
{\rm In \cite{BBdN1}, the main role is played by 
$$
\T\equiv\Cl_{\rm inv}(X):=\{M\in\Cl(X)\mid M\ {\rm is\ invariant}\}\,.
$$ 
The resulting groupoid $\Gamma(\Xi)\equiv\Gamma_{\rm inv}(\Xi)$ is called \emph{the tautological groupoid of $\,\Xi$}\,. We will mostly be interested in other choices.}
\end{Remark}

Let us recall that {\it a continuous field of groupoids} is given by a continuous open surjection $p:\Si\to\T$, where $\Si$ is a locally compact groupoid over a unit space $\Si^{(0)}$ and $\T$ a topological space (say compact) such that
\begin{equation*}\label{constrictie}
p|_{\Si^{(0)}}\!\circ \d=p=p|_{\Si^{(0)}}\!\circ \r\,.
\end{equation*}
Then clearly, for every $t\in\T$, the fiber $\Si_t:=p^{-1}(\{t\})$ is a closed subgroupoid and $\Si=\sqcup_{t\in\T}\Si_t$\,. The following criterion will be needed below:

\begin{Lemma}\label{ciciri}
\cite[Prop.1.15]{Wi} A surjection $\varphi:A\to B$ between two topological spaces is open if and only if for every $a\in A$ and every convergent net $B\ni b_i\to\varphi(a)$ there exist a subnet $\big(b_{i_k}\big)_k\!\subset B$ and a net $(a_k)_k\subset A$ with $\varphi(a_k)=b_{i_k}$ for every $k$ and $a_k\to a$\,.
\end{Lemma}

In our case we introduce the additional map
\begin{equation*}
p:\Gamma\!_\T(\Xi)\to\T, \quad p(M,\xi)\!:=M.
\end{equation*}
The next result is slightly more general than \cite[Prop.~4.3.7]{Bec}.

\begin{Proposition}\label{kokardel}
The triple $\big(\Gamma\!_\T(\Xi),\T,p\big)$ defines a continuous field of groupoids.
\end{Proposition}

\begin{proof}
Let $p^{(0)}$ be the restriction of $p$ to $\Gamma\!_\T(\Xi)^{(0)}$. The equalities 
\begin{equation}\label{shotii}
p^{(0)}\!\circ \r_\T\!=p=p^{(0)}\!\circ\d_\T
\end{equation} 
clearly hold, and  the map $p$ is also trivially surjective. Since $p$ is the restriction to the subset $\Gamma\!_\T(\Xi)$ of the continuous first projection ${\rm pr}_1:\T\times\Xi\to\T$, it is continuous. 

\smallskip
So we are left with showing that $p$ is open. The composition of two open maps is open, and $\d_\T,\r_\T$ are known to be open, since $\Gamma\!_\T(\Xi)$ has a Haar system. Thus, by \eqref{shotii}, we only need to prove that $p^{(0)}\!:\Gamma\!_\T(\Xi)^{(0)}\to\Cl(X)$ is open.
Applying Lemma \ref{ciciri} to our case, what has to be checked is that {\it for every $M\in\T$ and $x\in M$, for every convergent net $M_i\overset{\Cl(X)}{\longrightarrow}M=p^{(0)}(M,x)$\,, there exist a subnet $\big(M_{i_j}\big)_{\!k}$ and a net $\big(x_j\in M_{i_j}\big)_{\!j}$ with $x_j\overset{X}{\longrightarrow}x$\,.} 
But this is part of the characterization of the condition $M_i\overset{\Cl(X)}{\longrightarrow}M$, by Lemma \ref{indiciaza}.
\end{proof}

A particularization of \cite[Th.5.5]{LR} gives

\begin{Proposition}\label{away}
For $g\in{\rm C}^*\big[\Gamma_\T(\Xi)\big]$ and $M\in\T$\,, write $g_M:=g\vert_{\Xi(M)}\in C_{\rm c}\big[\Xi(M)\big]$\,. 
\begin{enumerate}
\item[(i)]
The map $\T\ni M\to\,\p\!g_M\!\p_{{\rm C}^*[\Xi(M)]}\,\in\R_+$ is upper semicontinuous.
\item[(ii)]
The map $\T\ni M\to\,\p\!g_M\!\p_{{\rm C}^*_{\rm r}[\Xi(M)]}\,\in\R_+$ is lower semicontinuous.
\end{enumerate}
\end{Proposition}

We formulated the result this way as a hint to the general case. For us, since each $\Xi(M)$ is amenable, there is no difference between the two $C^*$-norms and thus $M\to\,\p\!g_M\!\p_{{\rm C}^*[\Xi(M)]}$ is continuous. 

\begin{Remark}\label{traversura}
{\rm \medskip
In \cite{LR} second countable groupoids are treated. Having in view applications to  discrete metric spaces (Section \ref{grafologi}), which involve the Stone-\v Cech compactification, our $\sigma$-compactness assumption is preferable. For our spectral applications, we will only need the upper semicontinuity part of Proposition \ref{away} (see \eqref{reallygets}).
The proof of point 1 in [LR] relies on general facts, on Corollary 1.9 and Lemma 1.10 from \cite{Bl} (general statements not involving groupoids) and on Proposition 5.1 in the same [LR], that is a statement about short exact sequences. The initial statement comes from \cite{MRW}, where the second countability assumption holds over the entire article. However, in \cite[pag.\,334]{HLS} it is stated that the short exact sequence for full groupoid algebras (associated to a closed invariant subset of units) do hold without second countability. References as \cite{AZ} and \cite{SWi} also make use of the short exact sequence for $\sigma$-compact groupoids that are not second countable.}
\end{Remark}

\begin{Remark}\label{obstacle}
{\rm When transcribing Theorem 5.5 from \cite{LR} in our setting, we identified each fiber 
$$
\big[\Gamma_\T(\Xi)\big](M):=p^{-1}(\{M\})=\{(M,\xi)\mid\xi\in\Xi(M)\}=\{M\}\!\times\!\Xi(M)
$$
with $\Xi(M)$\,. Actually this fiber is a closed groupoid and it is the reduction of $\Gamma_\T(\Xi)$ to the closed subset $\{M\}\!\times\!M\subset\Gamma_\T(\Xi)^{(0)}$, {\it which is $\Gamma_\T(\Xi)$-invariant even when $M\subset X$ is not $\Xi$-invariant} (easy to check)! On the other hand, the map 
\begin{equation*}\label{iota}
\iota:{\rm C}^*(\Xi)\to{\rm C}^*\big[\Gamma_\T(\Xi)\big]\,,\quad[\iota(f)](M,\xi):=f(\xi)\,,
\end{equation*}
implicitly used when stating Proposition \ref{away}, is not a morphism (once again formula \eqref{banati} is relevant). One of the consequences is that it does not commute with the functional calculus: in general $\varphi(f)_M\ne\varphi(f_M)$\,.
This makes some strategies to prove spectral continuity unavailable. But this has no effect on the proof provided in the next section.
}
\end{Remark}

\section{Proof of Theorem \ref{dorita}}\label{stendhal}

Let us assume Setting \ref{sett} and adopt the notations of Theorem \ref{dorita}. One also uses notations as $\star_M$ and $\p\!\cdot\!\p_M$ for the product and for the norm, respectively, in the $C^*$-algebra ${\rm C}^*\big[\Xi(M)\big]$, with $M\in\T_0$ or $M=X_\infty$\,. We are also going to make use of the restrictions 
$$
\rho_{M\infty}\!:{\rm C}^*\big[\Xi(M)\big]\to{\rm C}^*\big[\Xi(X_\infty)\big]\,,\quad M\in\T_0\,.
$$ 
Since $X_\infty$ is a closed invariant set of units in $\Xi(M)$\,, this is a $C^*$-epimorphism. It follows easily that
\begin{equation}\label{legatura}
\rho_{M\infty}\circ\rho_{M}=\rho_{\infty}\,,\quad\forall\,M\in\T_0\,,
\end{equation}
which may also be written as $f_\infty=\rho_{M\infty}(f_M)$\,, if $f\in{\rm C}^*[\Xi]$\,. This implies in particular 
$$
\p\!f_\infty\!\p_{X_\infty}\,\le\!\underset{M\in\T_0}{\inf}\!\!\p\!f_M\!\p_{M}.
$$
Applying the upper semicontinuity of Proposition \ref{away} at the accumulation point $X_\infty$\,, one really gets
\begin{equation}\label{reallygets}
\p\!f_\infty\!\p_{X_\infty}\,=\!\underset{M\in\T_0}{\inf}\!\!\p\!f_M\!\p_{M}.
\end{equation}

(i) It is obvious that $C^*$-algebraic morphisms only decrease spectra (where ``unmodified'' is a particular case of ``decreased''). Applying this to $\rho_{M\infty}$ and using \eqref{legatura}, one gets
$$
{\sf sp}\big(h_\infty\big)={\sf sp}\big[\rho_{\infty}(h)\big]={\sf sp}\big[\rho_{M\infty}\big(\rho_{M}(h)\big)\big]\subset{\sf sp}\big(h_M\big), 
$$
which proves the inclusion 
\begin{equation}\label{doarjuma}
{\sf sp}\big(h_\infty\big)\subset\bigcap_{M\in\T_0}{\sf sp}\big(h_M\big)\,.
\end{equation} 
Note that if $h$ is self-adjoint, then each $h_M$ is self-adjoint (since $\rho_M$ is still involutive) and all the spectra are real. Then use \eqref{doarjuma} and the fact that for any family $\big\{\Lambda_M\!\mid\!M\in\T_0\big\}$ of subsets of $\R$ one may write 
\begin{equation}\label{lungutele}
\inf\!\Big(\underset{M}{\cap}\Lambda_M\Big)\ge\sup_M\!\big(\inf\Lambda_M\big)\quad{\rm and}\quad\sup\!\Big(\underset{M}{\cap}\!\Lambda_M\Big)\le\inf_M\big(\sup\Lambda_M\big)
\end{equation}
to get corresponding inequalities instead of \eqref{frik} and \eqref{frig}.

\smallskip
(ii) The point (i) only supplied inequalities; we included it in order to show that part of the result do not depend on continuity issues.
Let $h\in{\rm C}^*[\Xi]$ be self-adjoint; since all the restriction maps $\rho_M$ are contractive, the elements $h_M$ have a common lower bound $-c$\,, so $f\!:=h+c\mathfrak 1$ is a positive element and 
$$
f_M\!:=\rho_{M}(f)=h_M+c\mathfrak 1_M\in{\rm C}^*\big[\Xi(M)\big]\,,\quad M\in\T_0\,,
$$ 
$$
f_\infty\!:=\rho_{\infty}(f)=h_\infty+c\mathfrak 1_\infty\in{\rm C}^*\big[\Xi(X_\infty)\big]
$$ 
are all self-adjoint and positive. For normal elements of $C^*$-algebras the norm coincide with the spectral radius, which equals the supremum of the spectrum in the positive case. By \eqref{reallygets}, one infers that
\begin{equation*}\label{pazal}
\sup{\sf sp}\big(f_\infty\big)=\,\p\!f_\infty\!\p_{X_\infty}\,=\inf_{M\in\T_0}\!\p\!f_M\!\p_{M}\,=\!\inf_{M\in\T_0}\!\big[\sup {\sf sp}\big(f_M\big)\big].
\end{equation*}
Finally, since ${\sf sp}\big(f_\infty\big)={\sf sp}\big(h_\infty\big)+c$ and ${\sf sp}\big(f_M\big)={\sf sp}\big(h_M\big)+c$\,, one gets the equality \eqref{frig}.

\smallskip
Then, by applying this to $-h_M=-\rho_{M}(h)=\rho_{M}(-h)$\,, one obtains
$$
\begin{aligned}
\inf{\sf sp}\big(h_\infty\big)&=-\sup\!\big[{\sf -sp}\big(h_\infty\big)\big]=-\sup\!\big[{\sf sp}\big(\!-h_\infty\big)\big]\\
&=-\inf_{M\in\T_0}\!\big[\sup{\sf sp}\big(\!-h_M\big)\big]=-\inf_{M\in\T_0}\!\big[\sup\big(\!-{\sf sp}\big(h_M\big)\big)\big]\\
&=-\inf_{M\in\T_0}\!\big[\!-\inf\big({\sf sp}\big(h_M\big)\big)\big]=\sup_{M\in\T_0}\!\big[\inf{\sf sp}\big(h_M\big)\big],
\end{aligned}
$$
which finishes the proof.

\section{The represented version}\label{balanganit}

Let $\Xi$ be a manageable groupoid over the unit space $X$ and let $N$ be a tame subset of $X$. Recall that on the $\d$-fibres $\Xi_x$\,, $x\in X$ and $\Xi(N)_x=\Xi^N_x$, $x\in N$ one works with the counting measures, giving rise, respectively, to Haar systems. We are going to need the Hilbert spaces $\H_x\!:=\ell^2(\Xi_x)$ and $\H^N_x\!:=\ell^2\big(\Xi^N_x\big)$\,. Since for $x\in N$ one has $\Xi^N_x\subset\Xi_x$\,, we get the orthogonal decomposition $\H_x=\H^N_x\oplus\ell^2\big(\Xi_x\!\setminus\!\Xi_x^N\big)$\,. If $N$ is also invariant, clearly $\Xi^N_x=\Xi_x$ and $\H^N_x=\H_x$ for every $x\in N$. In fact, using for $N=X$ the conventions $\Xi^X_x\!:=\Xi_x$ and $\H^X_x\!:=\H_x$\,, one gets unified notations.

\smallskip
Let us fix $h\in C_{\rm c}(\Xi)\subset {\sf C}^*[\Xi]$\,, to which we associate operators $\big\{{\sf H}^N_x\in \mathbb B\big(\H^N_x\big)\,\big\vert\,x\in N\big\}$\,, given by
\begin{equation}\label{aram}
\big({\sf H}^N_x v\big)(\xi):=\!\!\sum_{\alpha\beta=\xi,\r(\beta)\in N}\!\!h(\alpha)v(\beta)=\!\sum_{\eta\in\Xi^N_x}\!h\big(\xi\eta^{-1}\big)v(\eta)\,,\quad v\in\ell^2\big(\Xi^N_x\big)\,,\ \xi\in\Xi^N_x.
\end{equation}
Note that this only involves the restriction of $h$ to $\Xi^N_N\equiv\Xi(N)$\,. 

\begin{Remark}\label{grinoble}
{\rm Setting $j^N_x\!:\H^N_x\to\H_x$ for the canonical injection and $r^N_x\!=\big(j^N_x\big)^*\!:\H_x\to\H^N_x$ for the canonical projection\,, one quickly checks that 
\begin{equation*}\label{arnaut}
{\sf H}^N_x=r^N_x\!\circ{\sf H}_x\!\circ j^N_x,\quad\forall\ x\in N,
\end{equation*}
so ${\sf H}^N_x$ is the (non-invariant) restriction of ${\sf H}_x$ to $\H^N_x$. It could be called {\it a compression}.}
\end{Remark}

\smallskip
We continue within the Setting \ref{sett}, using the same notations. In particular, we still suppose that $X_\infty$ is a closed invariant subset of $X$. Instead of ${\sf H}^{X_\infty}_x$ we write simply ${\sf H}^{\infty}_x$.

\begin{Theorem}\label{constiente}
Admitting the Setting \ref{sett}, let us assume, in addition, that the groupoid $\Xi$ is second countable. Let $h\in C^*[\Xi]$ be self-adjoint.  One has
\begin{equation}\label{ploicika}
\inf_{x\in X_\infty}\!\inf{\sf sp}\big({\sf H}^\infty_x\big)=\sup_{M\in\T_0}\!\Big[\inf_{y\in M}\inf {\sf sp}\big({\sf H}_y^M\big)\Big]\,,
\end{equation}
\begin{equation}\label{zapadika}
\sup_{x\in X_\infty}\!\sup{\sf sp}\big({\sf H}_x^\infty\big)=\inf_{M\in\T_0}\!\Big[\sup_{y\in M}\sup {\sf sp}\big({\sf H}_y^M\big)\Big]\,.
\end{equation}
\end{Theorem}

\begin{proof}
(i) Let us first prove that
\begin{equation}\label{farstate}
{\sf sp}\big(h_N\big)=\bigcup_{z\in N}\!{\sf sp}\big({\sf H}^N_z\big)\,,
\end{equation}
where $h_N=\rho_N(h)\in{\sf C}^*[\Xi(N)]$ is the restriction of $h$ to $\Xi(N)$\,. This has to be aplyed to $N=M,X_\infty$\,.

\smallskip
For every unit $z\in N$ there is \cite{Re,Pa,Wi} {\it the associated regular representation} $\pi^N_z\!:C_{\rm c}[\Xi(N)]\to\mathbb B\big[\ell^2(\Xi^N_z)\big]$\,, induced from the Dirac measure $\delta_z$ on $X$ and explicitly defined by
\begin{equation*}\label{bidreu}
\pi^N_z\!(f)v=f\star v\,,\quad v\in \ell^2\big(\Xi^N_z\big)\,.
\end{equation*}
It extends to a representation of the the groupoid $C^*$-algebra. It is known from \cite{Ex1} (see also \cite{NP} for a more general result) that 
\begin{equation}\label{fartate}
{\sf sp}(f)=\bigcup_{z\in N}\!{\sf sp}\big(\pi^N_z\!(f)\big)\,,
\end{equation}
a remarkable fact being that no closure is needed in the right hand side. The equality \eqref{fartate} holds even for $f\in {\sf C}^*[\Xi(N)]$\,. Here the second-countability condition is needed. A straightforward verification leads to the identification ${\sf H}^N_z\!=\pi^N_z\!(h_N)$ and this finishes our argument.

\smallskip
(ii) Then our result follows from the equalities \eqref{frik} and \eqref{frig}, from \eqref{farstate} and and from the formulae
$$
\inf\cup_i A_i=\inf_i\inf A_i\,,\quad \sup\cup_i A_i=\sup_i\sup A_i\,,
$$
valid for families of real sets. For example:
$$
\begin{aligned}
\inf_{x\in X_\infty}\!\inf{\sf sp}\big({\sf H}^\infty_x\big)&=\inf\Big[\bigcup_{x\in X_\infty}{\sf sp}\big({\sf H}_x^\infty\big)\Big]=\inf\big[{\sf sp}\big(h_\infty\big)\big]\\
&=\sup_{M\in\T_0}\!\big[\inf{\sf sp}\big(h_M\big)\big]=\sup_{M\in\T_0}\!\Big[\inf\bigcup_{y\in M}{\sf sp}\big({\sf H}^M_y\big)\Big]\\
&=\sup_{M\in\T_0}\!\Big[\inf_{y\in M}\inf{\sf sp}\big({\sf H}^M_y\big)\Big]\,.
\end{aligned}
$$
\end{proof}

\begin{Remark}\label{documenta}
{\rm Assume that $X_\infty$ is {\it topologically transitive}, i.\,e. that there is a dense orbit $\mathcal O$\,; this will lead to some simplifications in Theorem \ref{constiente}. Let $x_0$ be a point belonging to this orbit. It is known that for any other point $x_1$ of this orbit, the operators ${\sf H}^\infty_{x_0}$ and ${\sf H}^\infty_{x_1}$ are unitarily equivalent, thus isospectral. In addition, if $x\in X_\infty\!\setminus\mathcal O$, one has ${\sf sp}\big({\sf H}_x^\infty\big)\subset{\sf sp}\big({\sf H}_{x_0}^\infty\big)$\,. In particular, if $X_\infty$ is {\it minimal} (all the orbits are dense), then the spectrum of ${\sf H}^\infty_x$ is constant over $X_\infty$\,. Such results are quite common in the literature; see \cite{Ma} for instance and references therein.
Anyhow, returning to the topological transitive case, it follows that the l.\,h.\,s. of \eqref{ploicika} can be replaced by $\inf{\sf sp}\big({\sf H}^\infty_{x_0}\big)$ and the l.\,h.\,s. of \eqref{zapadika} by $\sup{\sf sp}\big({\sf H}^\infty_{x_0}\big)$\,.
}
\end{Remark}

\section{The essential spectrum}\label{talanganit}

This section is written in such a way that the reader could skip Section \ref{balanganit}. We go on assuming the Setting \ref{sett}. In dealing with the essential spectrum, we will need an extra assumption:

\begin{Definition}\label{zdandard}
The manageable groupoid $\Xi$ is called {\rm standard} if $\,X_0:=X\!\setminus\!X_\infty$ is a dense orbit whose isotropy is trivial, i.\,e.\;the isotropy group $\,\Xi_z^z$ is reduced to $\{z\}\,$ for some (and then every) $z\in X_0$ and if the (transitive and principal) groupoid $\,\Xi(X_0)$ is second countable\,.
\end{Definition}

 Then $X$ is a compactification of $X_0$\,, this one being called {\it the main orbit}\,. We only require the open subset $\Xi(X_0)$ to be second countable in the subspace topology; no such condition is imposed on $\Xi$\,.

\begin{Lemma}\label{atuncy}
The restricted map 
\begin{equation}\label{rd}
(\r,\d):\Xi(X_0)\to X_0\!\times\!X_0\,,\quad(\r,\d)(\xi):=\big(\r(\xi),\d(\xi)\big)
\end{equation}
is a homeomorphism.
\end{Lemma} 

\begin{proof}
Surjectivity follows from the fact that $X_0$ is an orbit, while injectivity follows from the triviality of the isotropy groups. Openness is difficult to prove, but this is done (in a more general context, using the second countability condition) in \cite[Th.\,2.2.A\,,Th.\,2.2.B]{MRW}.
\end{proof}

It is convenient to use the notation
\begin{equation}\label{rdinv}
\th:=(\r,\d)^{-1}\!:X_0\!\times\!X_0\to \Xi(X_0)\,.
\end{equation}
Thus, if $x,y\in X_0$\,, then $\th(x,y)$ is the unique element $\xi$ with $\r(\xi)=x$ and $\d(\xi)=y$\,. If $h\in C_{\rm c}(\Xi)$ one defines in the Hilbert space $\ell^2(X_0)$ the operator 
\begin{equation}\label{vinalla}
({\sf H}u)(x):=\sum_{y\in X_0}\!h\big(\th(x,y)\big)u(y)\,,\quad x\in X_0\,.
\end{equation}

Let us now assume that $M\in\T_0:=\T\!\setminus\!\{X_\infty\}$\,. A direct application of the definitions shows that $M_0\!:=M\cap X_0\subset M$ is an open orbit of the groupoid $\Xi(M)$ and that the isotropy over points of $M_0$ is trivial. We require, in addition, that $X_0\!\setminus\!M_0=X\!\setminus\!M$ is relatively compact in $X_0$\,. Then $M_0$ is dense in $M$ and it becomes the main orbit of $\Xi(M)$\,. {\it In such a case, $\Xi(M)$ is manageable and standard.} An argument similar to that indicated above associates to every $h_M\in C_{\rm c}\big(\Xi(M)\big)$ the operator 
\begin{equation}\label{vanilla}
\big({\sf H}_M v\big)(x):=\sum_{y\in M_0}h_M\big(\th_M(x,y)\big)v(y)\,,\quad x\in M_0\,,\ v\in\ell^2(M_0)\,,
\end{equation}
where $\th_M\!:=(\r,\d)^{-1}\!:M_0\!\times\!M_0\to \Xi(M_0)$ is the restriction of the map \eqref{rdinv}. 

\begin{Setting}\label{setter} 
Let $\Xi$ be a standard manageable groupoid with unit space $X$ and main orbit $X_0$\,. Let $\,\T$ be a closed subfamily of $\,\Cl(X)$ such that
\begin{itemize}
\item
the closed invariant subset $X_\infty\!:=X\!\setminus\!X_0$ is an accumulation point of $\,\T$\,,
\item
for every $M\in\T_0:=\T\!\setminus\!\{X_\infty\}$\,, one has $X_\infty\subset\!M$ and $X\!\setminus\!M$ is relatively compact in $X_0$\,.
\end{itemize}
\end{Setting}

In Remark \ref{cumetrii} there are some comments upon this setting, explaining in particular why all the subsets $M$ are tame. With respect to Setting \ref{sett}, we ask additionally $\Xi$ to be standard; the main orbit is the complement of the set $X_\infty$ and we require it to be open and free (no isotropy). In particular, $X$ is a {\it regular compactification} of $X_0$\,. For the next result, we denote by ${\sf sp_{\rm ess}}(S)$ {\it the essential spectrum} of the linear bounded operator $S$ acting in a Hilbert space. The statement will refer to compactly supported elements, but in Remark \ref{vine} more general choices will be discussed.

\begin{Theorem}\label{consecinte}
Admitting the Setting \ref{setter}, for any self-adjoint element $h\in C_{\sf c}(\Xi)$ and $M\in\T_0$ we consider the operators given by \eqref{vinalla} and \eqref{vanilla}, where $h_M\!:=h|_{\Xi(M)}$\,.
One has
\begin{equation}\label{ploicica}
\inf{\sf sp_{\rm ess}}({\sf H})=\sup_{M\in\T_0}\!\big[\inf {\sf sp}\big({\sf H}_M\big)\big]\,,
\end{equation}
\begin{equation}\label{zapadica}
\sup{\sf sp_{\rm ess}}({\sf H})=\inf_{M\in\T_0}\!\big[\sup {\sf sp}\big({\sf H}_M\big)\big]\,.
\end{equation}
\end{Theorem}

As a preparation for the proof, we are going to see how one recovers formulae as \eqref{vinalla} and \eqref{vanilla} by means of representations of groupoid $C^*$-algebras. First, for every unit $z$\,, we recall {\it the associated regular representation} $\pi_z\!:{\sf C}^*[\Xi]\to\mathbb B\big[\ell^2(\Xi_z)\big]$\,, defined slightly formally on $C_{\rm c}(\Xi)$ by
\begin{equation*}\label{bideu}
\pi_z(f)u=f\star u\,,\quad f\in C_{\rm c}(\Xi)\,,\ u\in \ell^2(\Xi_z)\,.
\end{equation*}
It has also been mentioned before, during the  proof of Theorem \ref{constiente}; its detailed definition as a representation induced from the Dirac measure $\delta_z$ on the unit space may be found in \cite[pag.\,16-18]{Wi1}. See also the proof of Lemma \ref{ingras}, where it will be used explicitly. It is known \cite{KS} that, whenever $z$ belongs to a dense orbit, $\pi_z$ is faithful. In fact this holds for the reduced algebra, but our amenability assumption settles this.

\smallskip
To switch to operators acting in $\ell^2(X_0)$\,, an important remark is needed. For every $z\in X_0$\,, the restriction 
$$
{\rm r}_z\!:={\rm r}|_{\Xi_z}\!:\Xi_z=\Xi(X_0)_z\to X_0
$$ 
of the range map to the $\d$-fiber over $z$\,, always surjective (since $X_0$ is an orbit), is also injective since the isotropy over $z$ is trivial. Clearly the transformation
$$
R_z\!:\ell^2(X_0)\to \ell^2\big(\Xi_z\big)\,,\quad R_z(u):=u\circ {\rm r}_z
$$
is a unitary operator. Thus one has a faithful representation 
\begin{equation*}\label{feisfulica}
\Pi:{\rm C}^*[\Xi]\to\mathbb B\big[\ell^2(X_0)\big]\,,\quad\Pi(f)=R_z^{-1}\pi_z(f)R_z\,,
\end{equation*}
 given on $C_{\rm c}(\Xi)$ by
\begin{equation*}\label{bibidem}
\Pi(f)u:=\big[f\star\!(u\circ {\rm r}_z)\big]\circ {\rm r}_z^{-1}.
\end{equation*}
We are going to see that it does not depend on $z$\,. Similar considerations may be applied to the standard groupoid $\Xi(M)$\,, using some $z\in M_0$\,; the corresponding representation will be denoted by $\Pi_M$. Generically, such representations are called {\it vector representations}. 

\begin{Lemma}\label{ingras}
In the framework of Theorem \ref{consecinte} one has ${\sf H}=\Pi(h)$ and ${\sf H}_M=\Pi_M(h_M)$\,.
\end{Lemma}

\begin{proof}
We compute for $x\in X_0$
$$
\begin{aligned}
\big(\Pi(h)u\big)(x)&=\big[h\star\!(u\circ {\rm r}_z)\big]\big({\rm r}_z^{-1}(x)\big)\\
&=\underset{{\eta\in\Xi_z}}{\sum}h\big({\rm r}^{-1}_z(x)\eta^{-1}\big)u\big(r_z(\eta)\big)\\
&=\underset{{y\in X_0}}{\sum}h\Big(\r^{-1}_z(x)\big[\r^{-1}_z(y)\big]^{-1}\Big)u(y)\\
&=\underset{{y\in X_0}}{\sum}h\big(\th(x,y)\big)u(y)\,.
\end{aligned}
$$
For the second equality we used the explicit form of the composition $\star$ and the identity 
$$
\r(\eta)^{-1}\!=\d(\eta)=\d\big[{\rm r}_z^{-1}(x)\big]=z\,.
$$ 
We get the third one through a change of variables. The forth follows from 
$$
\d\Big(\r^{-1}_z(x)\big[\r^{-1}_z(y)\big]^{-1}\Big)=y\quad{\rm and}\quad\r\Big(\r^{-1}_z(x)\big[\r^{-1}_z(y)\big]^{-1}\Big)=x\,,
$$ 
by unicity. A similar calculation leads to ${\sf H}_M=\Pi_M(h_M)$\,.
\end{proof}

\begin{Remark}\label{cumetrii}
{\rm We noticed above that the continuous map ${\rm r}_z\!:\Xi_z\to X_0$ is one-to-one. But, as follows from \cite[Th.\,2.2.A\,,Th.\,2.2.B]{MRW}, it is also open; second-countability of $\Xi(X_0)$ is used. Recalling Remark \ref{ital}, we conclude that {\it $X_0$ and $\Xi_z$ are homeomorphic discrete countable spaces}. Therefore, asking $X\!\setminus\!M$ to be relatively compact in Setting \ref{setter} actually means that it is finite. Under such a hypothesis $M$ is clopen and, by Lemmas \ref{daca} and \ref{teim}, it is tame and $\Xi(M)$ is automatically \'etale.}
\end{Remark}

We proceed now to the proof of Theorem \ref{consecinte}.

\begin{proof}
The (invariant) reduction $\Xi(X_0)$ to the orbit $X_0$ is isomorphic to the pair groupoid $X_0\!\times\!X_0$ through \eqref{rd}\,. And then, by \cite[Th.\,3.1]{MRW}, one has ${\rm C}^*[\Xi(X_0)]\cong\mathbb K\big[\ell^2(X_0)\big]$\,. Similarly, $\Xi(M_0)$ is isomorphic to $M_0\!\times\!M_0$ and ${\rm C}^*[\Xi(M_0)]\cong\mathbb K\big[\ell^2(M_0)\big]$\,. The isomorphisms are precisely realized by the corresponding vector representations.

\smallskip
By Lemma \ref{ingras}, and since each $\Pi_M$ is faithful, one has ${\sf sp}\big({\sf H}_M\big)={\sf sp}\big[\Pi_M\big(h_M)\big]={\sf sp}(h_M)$\,. This shows that the right-hand sides of \eqref{frik} and \eqref{ploicica} and of \eqref{frig} and \eqref{zapadica}\,, respectively, are equal, so we only need to show that the left-hand sides are also equal. The essential spectrum of ${\sf H}=\Pi(h)$ coincides with the spectrum of its canonical image in the Calkin $C^*$-algebra $\mathbb B\big[\ell^2(X_0)\big]/\mathbb K\big[\ell^2(X_0)\big]$\,; this image is actually contained in 
$$
\Pi\big({\rm C}^*[\Xi]\big)/\mathbb K\big[\ell^2(X_0)\big]=\Pi\big({\rm C}^*[\Xi]\big)/\Pi\big[\ker(\rho_{\infty})\big]\cong{\rm  }{\rm C}^*[\Xi]/\ker(\rho_{\infty})\cong{\rm C}^*\big[\Xi(X_\infty)\big]\,.
$$
We used the fact (cf.\;\cite{MRW}) that ${\sf C}^*[\Xi(X_0)]$ can be identified with the kernel of the epimorphism $\rho_\infty:{\rm C}^*[\Xi]\to{\rm C}^*\big[\Xi(X_\infty)\big]$\,.
Under the above composition of isomorphisms (that all preserve the spectra), ${\sf H}+\mathbb K\big[\ell^2(X_0)\big]$ corresponds to $h_\infty=\rho_{\infty}(h)$\,, so ${\sf sp}_{\rm ess}({\sf H})={\sf sp}(h_\infty)$\,.
This finishes the proof.
\end{proof}

\begin{Remark}\label{compression}
{\rm One can write $\ell^2(X_0)=\ell^2(M_0)\oplus\ell^2(X_0\!\setminus\!M_0)$\,. In terms of the canonical injection ${\sf j}_{M}:\ell^2(M_0)\to\ell^2(X_0)$ and its adjoint, the canonical projection ${\sf p}_{M}:\ell^2(X_0)\to\ell^2(M_0)$\,, we define 
\begin{equation*}\label{lichie}
\mathfrak R_{M}:\mathbb B\big[\ell^2(X_0)\big]\to\mathbb B\big[\ell^2(M_0)\big]\,,\quad \mathfrak R_{M}(S):={\sf p}_{M}\circ S\circ{\sf j}_{M}\,.
\end{equation*}
It is a linear contraction; 
{\it in general it is not a $C^*$-morphism.} Its relevance comes from the formula ${\sf H}_M=\mathfrak R_{M}({\sf H})$\,,
showing that {\it ${\sf H}_M$ is a compression of $\,{\sf H}$}\,. In fact, one has the commutative diagram
\begin{equation*}\label{iagram}
\begin{diagram}
\node{{\rm C}^*[\Xi]}\arrow{e,t}{\Pi}\arrow{s,r}{\rho_M}\node{\mathbb B\big[\ell^2(X_0)\big]}\arrow{s,r}{\mathfrak R_M}\\
\node{{\rm C}^*[\Xi(M)]}\arrow{e,t}{\Pi_M}\node{\mathbb B\big[\ell^2(M_0)\big]}
\end{diagram}
\end{equation*}
composed of two horizontal faithful representations and two vertical linear contractions, both modeling (non-invariant) restrictions. 
}
\end{Remark}

\begin{Remark}\label{vrine}
{\rm As said before, the present section is written independently of Section \ref{balanganit}. To understand why in Theorem \ref{consecinte} there is no proliferation of Hamiltonians connected to various units, as in Theorem \ref{constiente}, one has to take into account the vector representations and the fact that both $\Xi$ and $\Xi(M)$ have topologically transitive unit spaces. Remark \ref{documenta} could be reviewed at this point.}
\end{Remark}

\begin{Remark}\label{vine}
{\rm In Theorem \ref{consecinte} the element $h$ has been chosen compactly supported, for simplicity. For the nature of general elements $h\in{\rm C}^*[\Xi]$ (still functions) and their restrictions $h\to h_M$\,, we recall Remarks \ref{han} and \ref{onrestrictions}, respectively.  The operators $\Pi(h)$ and $\Pi_M(h_M)$ make sense in general, and this is all one needs in the proof of Theorem \ref{consecinte}. Let us make, however, some comments on instances when the explicit formulae \eqref{vinalla} and \eqref{vanilla} may still be used as they stand. One can apply the Schur test to operators of the form
\begin{equation*}\label{cuintegru}
[\Upsilon(F)u](x):=\sum_{y\in X_0}F(x,y)u(y)\,,
\end{equation*}
where $u\in\ell^2(X_0)$ and $F:X_0\times X_0\to\mathbb C$ is a suitable function, getting 
\begin{equation*}\label{schnur}
\p\!\Upsilon(F)\!\p_{\mathbb B[\ell^2(X_0)]}\,\le\Big(\sup_{x}\sum_y|F(x,y)|\Big)^{1/2}\Big(\sup_{y}\sum_x|F(x,y)|\Big)^{1/2}.
\end{equation*}
In terms of  \eqref{rdinv}, define the map $\Th(f):=f\circ\th$\,, for which one has $\Pi=\Upsilon\circ\Th$\,. One checks easily that 
$$
\th\big(\{x\}\times X_0\big)=\Xi^x\quad{\rm and}\quad \th\big(X_0\times\{y\}\big)=\Xi_y\,,\quad \forall\,x,y\in X_0\,.
$$ 
Then, by an obvious change of variables
$$
\begin{aligned}
\p\!\Pi(f)\!\p_{\mathbb B[\ell^2(X_0)]}\,&\le \Big(\sup_{x}\sum_y\big\vert f\big(\th(x,y)\big)\big\vert\Big)^{1/2}\Big(\sup_{y}\sum_x\big\vert f\big(\th(x,y)\big)\big\vert\Big)^{1/2}\\
&=\Big(\sup_{x}\sum_{\xi\in\Xi^x}|f(\xi)|\Big)^{1/2}\Big(\sup_{y}\sum_{\xi\in\Xi_y}|f(\xi)|\Big)^{1/2}\\
&\le\,\max\Big(\sup_{x}\sum_{\xi\in\Xi^x}|f(\xi)|\,,\sup_{y}\sum_{\xi\in\Xi_y}|f(\xi)|\Big)\\
&= \;\p\!f\!\p_{{\rm Hahn}(\Xi)}\,,
\end{aligned}
$$
showing that $\Pi:{\rm Hahn}(\Xi)\to\mathbb B\big[\ell^2(X_0)\big]$ is contractive (this also follows from abstract considerations) and that $\Pi(h)$ is given by formula \eqref{vinalla} for every $h$ in the Hahn algebra introduced in Remark \ref{han}. Then $X_0$ may be replaced by $M_0$\,.
}
\end{Remark}

\section{Bi-sided subshifts}\label{geografi}

Let $\A$ be a finite set (an alphabet) seen as a discrete space, with elements often called {\it letters}. The space $\A^\Z$ of all functions $\Z\to\A$ is a second countable, compact, totally disconnected, metrizable space with the product topology. The group $\Z$ acts continuously on $\A^\Z$ by
\begin{equation*}\label{ctiune}
\tau:\Z\!\times\!\A^\Z\to\A^\Z,\quad\tau_k(x):=x(\cdot-k)\,.
\end{equation*}

We often regard the elements of $\A^Z$ as two-sided sequences and use notations as
\begin{equation*}\label{tations}
x=\dots a_{-3}a_{-2}a_{-1}\!\downarrow\!a_0a_1a_2a_3\dots
\end{equation*}
with $a_j\in\A$\,, the arrow indicating the position of $0$\,. To indicate a base of the topology of $\A^\Z$, and for other reasons, we are going to use {\it (finite) words}, i.\,e.\;elements of {\it the free monoind} $\A^*\!:=\bigsqcup_{n\in\N}\A^n$\,, where $\A^0\!:=\{\heartsuit\}$ consists only of the empty word. If $u\in\A^n$ we write $\#u=n$\,; thus $\#\heartsuit=0$\,. The monoid composition law on $\A^*$ is concatenation. If $v=w_1 uw_2$\,, for two words $w_1,w_2$\,, we write $u\prec v$ and say that {\it $u$ is contained in} $v$. For $x\in\A^\Z$ and $i,j\in\Z$ with $i\le j$\,, we denote by $x|_{[i,j]}$ the restriction of $x$ to the "interval" $[i,j]:=\{i,i+1,\dots,j-1,j\}$\,, seen as a word $u$ with letters in $\A$\,. We also write $u\prec x$\,. One could also write $x|_F$ for the restriction of $x$ to the subset $F$ of $\Z$\,. The family 
\begin{equation*}\label{dictionar}
\mathcal D(x):=\{u\in\A^*\!\mid u\prec x\}
\end{equation*} 
is called {\it the dictionary of $x$}\,, and if $Y\subset\A^\Z$ then 
\begin{equation*}\label{superdictionar}
\mathcal D(Y)\!:=\bigcup_{x\in Y}\mathcal D(x)
\end{equation*} 
is called {\it the dictionary of $\,Y$}. 

\smallskip
We call {\it (bi-sided) subshift} any closed $\tau$-invariant subset $X$ of $\A^\Z$\,; then $(X,\tau,\Z)$ is a dynamical system. The space $X$ with the induced topology is second countable, Hausdorff, compact, totally disconnected and metrizable, i.\,e.\;{\it a second-countable Stone space}. A countable base of clopen sets is 
\begin{equation*}\label{baza}
\mathscr B:=\{\mathcal C_X(u,v)\!\mid\! uv\in\mathcal D(X)\}\,,
\end{equation*}
where, if $\#u=n$ and $\#v=m$\,, one introduces {\it the cylindrical set}
\begin{equation*}\label{clopen}
\mathcal C_X(u,v):=\big\{x\in X\,\big\vert\, x|_{[-n,m-1]}=u\!\downarrow\!v\big\}\,.
\end{equation*}
The convergence in $X$ is easy to express: 
{\it one has $x_\nu\to x$ if and only if for every $k\in\Z$ there exists $\nu_k$ such that $x_\nu(k)=x(k)$ for every $\nu\ge\nu_k$}\,. So the topology of pointwise convergence is involved, combined with the fact that $\A$ is a discrete (finite) space. To the dynamical system $(X,\tau,\Z)$ one associates canonically a transformation groupoid $\Xi:=X\!\rtimes_\tau\Z$\,. It will appear in detail in the proof of Theorem \ref{constiinte}, together with its non-invariant restrictions. We prefer to introduce first the operators we study and to state the result.

\smallskip
So we introduce now some objects and notations that will be used in Theorem \ref{constiinte}. Let us fix a closed subset $N$ of the bi-sided subshift $X\subset\A^\Z$\,; it could be non-invariant under $\tau$. For any $x\in N$ set 
\begin{equation}\label{turulec}
\Z_x(N):=\{k\in\Z\!\mid\!\tau_k(x)\in N\}\subset\Z\,.
\end{equation}
To any continuous function $h:X\!\times\!\Z\to\mathbb C$ having compact support, we attach the bounded operator ${\sf H}^N_x$ in the Hilbert space $\ell^2\big(\Z_x(N)\big)$ given by the formula
\begin{equation}\label{voinik}
\big[{\sf H}^N_x\!(v)\big](k):=\!\!\sum_{l\in\Z_x(N)} \!h\big(\tau_l(x),k-l\big)v(l)\,,\quad k\in\Z_x(N)\,.
\end{equation}

\begin{Remark}\label{manshelam}
{\rm We will say that {\it $h$ is self-adjoint} if
\begin{equation*}\label{sa}
h\big(\tau_k(x),-k\big)=\overline{h(x,k)}\,,\quad\forall\,x\in X,\ k\in\Z\,.
\end{equation*}
The groupoid meaning of this requirement will become obvious in the proof of Theorem \ref{constiinte}. For the time being, it is enough to note that it insures the self-adjointness of the operators appearing in \eqref{voinik}.
}
\end{Remark}

\begin{Example}\label{firstday}
{\rm  If $h$ does not depend on the first variable, it can be seen as a sort of Toeplitz operator: it is the compression to $\ell^2\big(\Z_x(N)\big)$ of the convolution operator by $h:\Z\to\mathbb C$ acting in $\ell^2(\Z)$\,.
But in general ${\sf H}^N_x$ is a much more complicated and interesting operator.}
\end{Example}

\begin{Example}\label{seconday}
{\rm
If (and only if) $N$ is $\tau$-invariant, one gets $\Z_x(N)=\Z$\,. This will be the case in Theorem \ref{constiinte} for $N=X$ or $N=X_\infty$\,. In such a case, one may conjugate the operator ${\sf H}^N_x$, acting in $\ell^2(\Z)$\,, with the standard unitary Fourier transformation $\mathcal F:\ell^2(\Z)\to L^2(\mathbb T)$\,, where the torus $\mathbb T$ is seen as the Pontryagin dual $\hat\Z$ of the Abelian group $\Z$\,. The operator ${\rm Op}\big(\tilde h_x\big)\!:=\mathcal F\,{\sf H}^N_x\mathcal F^{-1}\!\in\mathbb B\big[L^2(\mathbb T)\big]$ has the explicit form
\begin{equation*}\label{pseudodiff}
\big[{\rm Op}\big(\tilde h_x\big)w\big](t)=\sum_{l\in\Z}\int_\mathbb T e^{il(s-t)}\hat h(\tau_l(x),t)w(s)ds\,.
\end{equation*}
We denoted by $\hat h:N\!\times\!\mathbb T\to\mathbb C$ the partial Fourier transformation of $h$ in the second variable and set 
$$
\tilde h_x:\Z\times\hat\Z\to\mathbb C\,,\quad\tilde h_x(l,t):=\hat h\big(\tau_l(x),t\big)\,.
$$ 
Hence {\it ${\rm Op}\big(\tilde h_x\big)$ is a global pseudodifferential operator on the group $\Z$ with symbol depending on the parameter $x\in N$}, cf. \cite{MR} and references therein. The family of symbols $\big\{\tilde h_x\,\big\vert\,x\in N\big\}$ (and, consequently, the family of operators $\big\{{\sf H}_x^N\,\big\vert\,x\in N\big\}$) has a covariance property: 
$$
\tilde h_{\tau_m(x)}(\cdot,\cdot)=\tilde h_x(\cdot+m,\cdot)\,,\quad\forall\,m\in\Z\,,x\in X.
$$
}
\end{Example}

\begin{Theorem}\label{constiinte}
Let $X_\infty$ be a closed $\tau$-invariant subset of the bi-sided subshift $X$. Let $\{M_i\!\mid\! i\in I\}$ be a net of clopen subsets of $X$ such that $X_\infty\subset M_i\,$ for every $i$ and such that $M_i\!\to X_\infty$ in the Fell topology of $\,{\rm Cl}(X)$\,. Let $h\in C_{\sf c}(X\!\times\!\Z)$ be self-adjoint.  One has
\begin{equation}\label{ploicik}
\inf_{x\in X_\infty}\!\inf{\sf sp}\big({\sf H}^{X_\infty}_x\!\big)=\sup_{i\in I}\!\Big[\inf_{y\in M_i}\inf {\sf sp}\big({\sf H}_y^{M_i}\big)\Big]\,,
\end{equation}
\begin{equation}\label{zapadik}
\sup_{x\in X_\infty}\!\sup{\sf sp}\big({\sf H}_x^{X_\infty}\!\big)=\inf_{i\in I}\!\Big[\sup_{y\in M_i}\sup {\sf sp}\big({\sf H}_y^{M_i}\big)\Big]\,.
\end{equation}
\end{Theorem}

\begin{proof}
We apply Theorem \ref{constiente} for {\it the transformation groupoid} $\Xi:=X\!\rtimes_\tau\Z$ associated to the dynamical system $(X,\tau,\Z)$\,; for details we refer to \cite{Pa,Re,Wi1}. Just recall that, as a topological space, $\Xi=X\!\times\!\Z$ with the product topology. The groupoid structure is indicated by the formulae
\begin{equation}\label{turuliek}
\d(x,k):=x\,,\quad\r(x,k):=\tau_k(x)\,,
\end{equation}
$$
(\tau_l(x),k)(x,l):=(x,k+l)\,,\quad(x,k)^{-1}\!:=(\tau_k(x),-k)\,.
$$
Since $\Z$ is discrete and amenable, $\Xi$ is \'etale and amenable. Consequently, {\it the transformation groupoid $\Xi:=X\!\rtimes_\tau\Z$ is  manageable and second countable}. The invariant restriction $\Xi(X_\infty)$ is the transformation groupoid $X_\infty\!\rtimes_\tau\!\Z$\,. The non-invariant restrictions $\Xi(M_i)$ are only associated to {\it partial} dynamical systems \cite{Ab,Ex,Ex2}, but since this will not be used explicitly, we do not describe them. All the sets $N=X_\infty$ and $N=M_i$\,, $i\in I$ are tame, by Lemma \ref{daca}. 

\smallskip
We still have to identify the operators appearing in the statement. Note that for $x\in N\subset X$, the $\d$-fiber over $x$ in the groupoid $\Xi(N)$ is 
$$
\Xi(N)_x\equiv\Xi^N_x=\{(x,k)\!\mid\!\tau_k(x)\in N\}=\{x\}\!\times\!\Z_x(N)\,,
$$ 
which allows us to identify the Hilbert space $\H^N_x\!:=\ell^2\big(\Xi^N_x\big)$ with $\ell^2\big(\Z_x(N)\big)$ in the obvious way. A straightforward application of the definitions shows then that the operator \eqref{aram} becomes \eqref{voinik}. We conclude by Theorem \ref{constiente}.
\end{proof}

We recall \cite[pag.165]{dV} that one of the compatible metrics on the subshift $X$ is given by
\begin{equation*}\label{mietrica}
d(x,y):=\big(1+\min\big\{|k|\,\big\vert\, x_k\ne y_k\big\}\big)^{-1}\quad{\rm if}\ x\ne y\,.
\end{equation*}
The balls are cylindric sets. If for every $i\in\N$ one defines the closed set
\begin{equation}\label{baraie}
\begin{aligned}
M_i:&=\big\{y\in X\,\big\vert\,d(y,X_\infty)\le 1/(i+1)\big\}\\
&=\big\{y\in X\,\big\vert\,\exists\,x\in X_\infty\,,\,x_k=y_k\,,\,\forall\,|k|<i\big\}\,,
\end{aligned}
\end{equation}
then $M_i$ converges to $X_\infty$ in the Fell topology when $i\to\infty$ (use Lemma \ref{indiciaza} once again). Since the metric $d$ only takes values in the set $1/\N\cup\{0\}$\,, we have
$$
M_i=\big\{y\in X\,\big\vert\,d(y,X_\infty)< 1/(i+1)+\nu\big\}\quad{\rm if}\ \ 1/(i+1)+\nu<1/i
$$ 
and $M_i$ is also open. Therefore the family $\{M_i\!\mid\! i\in\N\}$ may be used in Theorem \ref{constiinte}.

\smallskip
We are going now to apply the results of Section \ref{talanganit} on the essential spectrum, defining subshifts as orbit closures of suitable sequences.

\begin{Definition}\label{admisibil}
The sequence $z\in\A^\Z$ is called {\rm admissible} if it is not periodic (as a function $z:\Z\to\A$) and its orbit $\mathcal O_z\!:=\tau_\Z(z)$ is open in the orbit closure $\overline{\mathcal O}_z$\,. In such a case we set
$$
X:=\overline{\mathcal O}_z\,,\quad X_0:=\mathcal O_z\,,\quad X_\infty:=\overline{\mathcal O}_z\!\setminus\!\mathcal O_z\,. 
$$
\end{Definition}

\begin{Remark}\label{adasta}
{\rm In terms of dictionaries, one has $\mathcal D(x)=\mathcal D(z)$ if $x\in\mathcal O_z$ (obvious) and \cite[pag.165]{dV}
$$
\overline{\mathcal O}_z=\big\{y\in\A^\Z\,\big\vert\,\mathcal D(y)\subset\mathcal D(z)\big\}.
$$
The orbit $\mathcal O_z$ is open in the orbit closure if and only if there are words $u,v$ such that $z\in\mathcal C_X(u,v)\subset\mathcal O_z$\,. Thus, for every $y$ with $\mathcal D(y)\subset\mathcal D(z)$ and $y|_{[-\#u,\#v-1]}=u\!\downarrow\!v$ we should have $y=\tau_k(z)$ for some $k\in\Z$\,. 
}
\end{Remark}

\begin{Example}\label{vascos}
{\rm As in \cite[Ex.1]{BBdN2}, we take $\A:=\{a,b\}$ and $z=\dots aabb\dots$\,. The orbit closure $X\!:=\overline{\mathcal O}_z$ of $z$ differs from the orbit by two fixed points $\dots aaaa\dots$ and $\dots bbbb\dots$\,, hence $z$ is admissible. Using the notation \eqref{baraie}, for any $i\in\N$ one gets 
$$
M_i=\{y\in X\!\mid\! y_k=a\,,\forall\,|k|<i\}\cup\{y\in X\!\mid\! y_k=b\,,\forall\,|k|<i\}\,.
$$
}
\end{Example}

\begin{Example}\label{vrascos}
{\rm As in \cite[pag.167]{dV}, we take $\A:=\{a,b\}$ and $z=\(z_k\)_{k\in\Z}$\,, with $z_k=b$ if and only if $|k|=2^m$ for some $m\in\N$ (some similar choices are possible). The orbit closure $X$ of $z$ is partitioned as $\overline{\mathcal O}_z=\mathcal O_z\sqcup\mathcal O_{z'}\!\sqcup\mathcal O_{z''}$\,, where $z'$ is the fixed point $\dots aaaaa\dots$ and $z''=\dots aabaa\dots$\,. It is clear that $\mathcal O_{z'}$ is closed and the closure of $\mathcal O_{z''}$ is $\mathcal O_{z'}\!\sqcup\mathcal O_{z''}$\,, so our point $z$ is admissible. Here we have $M_i=X_\infty\cup N_i'\cup N_i''$, where
$$
\tau_l(z)\in\ N_i'\,\Leftrightarrow\,z_{k-l}=a\,,\,\forall\,|k|<i\,,
$$
$$
\tau_l(z)\in\ N_i''\,\Leftrightarrow\,\exists\,m\in\Z\,,\ z_{k-l}=z''_{k-m}\,,\,\forall\,|k|<i\,.
$$
}
\end{Example}

Let $h\in C_{\sf c}(X\!\times\!\Z)$ be self-adjoint, cf. Remark \ref{manshelam}. If $X_\infty\subset M\subset X$, setting $M_0:=M\!\setminus\!X_\infty\subset X_0$\,, we introduce the bounded operator ${\sf H}_M$ in the Hilbert space $\ell^2\big(M_0\big)$ given by the formula
\begin{equation}\label{voinig}
\big[{\sf H}_M(v)\big](x):=\!\!\sum_{-l\in\Z_x(M)} \!h\big(\tau_{-l}(x),l\big)v\big(\tau_{-l}(x)\big)\,,\quad x\in M_0\,.
\end{equation}
It is easy to check that ${\sf H}_M$ is the compression of ${\sf H}\equiv{\sf H}_X\!\in\mathbb B\big[\ell^2\big(X_0\big)\big]$ to $\ell^2\big(M_0\big)$\,.

\begin{Corollary}\label{fillip}
Let $\{M_i\!\mid\! i\in I\}$ be a net of clopen subsets of $X$ such that $X_\infty\subset M_i$\,, such that $X\!\setminus\!M_i=X_0\!\setminus\!M_{i,0}$ is relatively compact in $X_0\,$ for every $i$ and such that $M_i\!\to X_\infty$ in the Fell topology of $\,{\rm Cl}(X)$\,. One has
\begin{equation*}\label{ploigica}
\inf{\sf sp_{\rm ess}}({\sf H})=\sup_{i\in I}\!\big[\inf {\sf sp}\big({\sf H}_{M_i}\big)\big]\,,
\end{equation*}
\begin{equation*}\label{zapadiga}
\sup{\sf sp_{\rm ess}}({\sf H})=\inf_{i\in I}\!\big[\sup {\sf sp}\big({\sf H}_{M_i}\big)\big]\,.
\end{equation*}
\end{Corollary}

\begin{proof}
By our assumptions, $X_\infty$ is a closed invariant subset of the subshift $X$. It is non-void, since only the orbits of $\Z$-periodic points can be closed.

\smallskip
Theorem \ref{consecinte} requires the groupoid $\Xi:=X\!\rtimes_\tau\!\Z$ to be standard, cf. Definition \ref{zdandard}. Clearly the orbits of the transformation groupoid coincide with the orbits of the dynamical system $(X,\tau,\Z)$\,, thus $X_0$ is an open dense orbit (the single one). It also clear that the isotropy group $\Xi^y_y$ is precisely $\{y\}\!\times\!\Z_{\{y\}}$, where $\Z_{\{y\}}\!:=\{k\in\Z\!\mid\! \tau_k(y)=y\}$ is the isotropy group of the point $y$ under the action $\tau$. These subgroups are all isomorphic along the orbit. Thus, by our requirements, we insured that $X_0$ is a dense open orbit with trivial anisotropy, as in Definition \ref{zdandard}, and Theorem \ref{consecinte} may be applied.

\smallskip
To finish the proof, one needs to identify correctly the Hamiltonians, i.\,e.\;to make the connection between \eqref{vanilla} and \eqref{voinig} for $M=X$ or $M=M_i$\,. Recall that $M_0$ is the main orbit of $\Xi(M_0)$\,. In our case, by \eqref{turuliek}, the function 
$$
\th_M\!:=(\r,\d)^{-1}\!:M_0\!\times\! M_0\to\Xi(M_0)\subset X\!\times\!\Z
$$
sends the pair $(x,y)$ to $(y,l)$\,, where $l\in\Z$ is the unique element such that $\tau_l(y)=x$\,. This means that $y=\tau_{-l}(x)$ and that actually $-l$ belongs to $\Z_x(M)$ (see \eqref{turulec}). When $x$ is fixed, the correspondence $y\leftrightarrow l$ is a bijection and \eqref{vanilla} is converted into \eqref{voinig}, since $h_M$ is just a restriction of $h$\,.
\end{proof}

The reader can easily verify the coherence between formulas \eqref{voinik} and \eqref{voinig} by using the vector representation outlined in the paragraph preceding Lemma \ref{ingras}.

\section{Discrete metric spaces}\label{grafologi}

In this section we will prove Theorem \ref{uvert} and make some extra comments. We are going to fix an infinite countable metric space $(X_0,\delta)$ satisfying
\begin{enumerate}
\item[(UD)] {\it uniform discreteness:} for some $\alpha>0$\,, one has $\delta(x,x')\ge\alpha$ if $x\ne x'$,
\item[(BG)]
{\it bounded geometry:} for every positive $r$, all the closed balls $B_x(r)$ of radio $r$ have less then $N_r<\infty$ elements.
\end{enumerate} 

 Linear bounded operators ${\sf H}$ in the Hilbert space $\ell^2(X_0)$ are given by (suitable) "$X_0\!\times\!X_0$-matrices":
\begin{equation}\label{fraier}
[{\sf H}(u)](x)=\sum_{y\in X_0}\!H(x,y)u(y)\,.
\end{equation}

\begin{Definition}\label{frayer}
\begin{enumerate}
\item[(i)]
We say that ${\sf H}\in\mathbb B\big[\ell^2(X_0)\big]$ is {\rm a band operator} if $\,\sup_{x,y}|H(x,y)|<\infty$ and if there is some positive $r$ such that $H(x,y)=0$ if $\,\delta(x,y)>r$\,.
\item[(ii)]
A linear bounded operator is {\rm band dominated} if it is a norm-operator limit of band operators.
\end{enumerate}
\end{Definition}

The band operators form a $^*$-algebra ${\rm Band}(X_0,\delta)$\,. Consequently, the band-dominated operators form a unital $C^*$-subalgebra ${\rm Roe}(X_0,\delta)$ of $\mathbb B\big[\ell^2(X_0)\big]$\,, called {\it the uniform Roe algebra of the metric space}. 

\smallskip
Let us denote by $\beta X_0\equiv X$ {\it the Stone-\u Cech compactification} of the discrete space $X_0$\,. It is a Stone space; it is even extremally disconnected (the closure of any open set is open). Set $X_\infty\!:=\beta X_0\!\setminus\!X_0$\,; then $X_0$ is a dense open subset of $\beta X_0$\,. In \cite{STY,Roe} a groupoid model has been given for the uniform Roe algebras.  For every $r\ge 0$ one sets 
$$
\Delta_r\!:=\{(x,x')\in X_0\!\times\!X_0\!\mid \delta(x,x')\le r\}\,,
$$ 
with closure $\overline{\Delta_r}$ in $\beta(X_0\!\times\!X_0)$\,, and then 
\begin{equation*}\label{abiert}
\Xi(X_0,\delta)\!:=\bigcup_{r>0}\,\overline{\Delta}_r\,.
\end{equation*}
It is the Gelfand spectrum of the $C^*$-subalgebra of $\ell^\infty(X_0\!\times\!X_0)$ generated by the charecteristic functions $\chi_{\Delta_r}$ with $r\ge 0$\,.
The two canonical projections $X_0\!\times\!X_0\to X_0$ extend to maps $\beta(X_0\!\times\!X_0)\to \beta X_0$ and then restrict to maps $\r,\d:\Xi(X_0,\delta)\to\beta X_0$\,, and it turns out that $(\r,\d):\Xi(X_0,\delta)\to\beta X_0\times \beta X_0$ is injective. 
Finally $\Xi(X_0,\delta)$ is identified to an open subset of $\beta X_0\!\times\!\beta X_0$\,, and the pair groupoid structure restricts to a groupoid structure on $\Xi(X_0,\delta)$\,. A topology on $\Xi(X_0,\delta)$ is defined as follows: the subset $U\subset \Xi(X_0,\delta)$ is open if and only if $U\cup\overline{\Delta}_r$ is open in $\overline\Delta_r$ for any $r>0$\,. Then $\,\Xi(X_0,\delta)$ {\it is a $\si$-compact \'etale standard groupoid over the unit space $\beta X_0$\,, with main orbit $X_0$}\,. In addition, its invariant restriction to $X_0$ coincides with the pair groupoid $X_0\!\times\!X_0$\,.
Such facts and the next result may be found in \cite[Sect.10.3,10.4]{Roe}; see also \cite[Ch.3]{STY} and \cite[App.C]{SWi}.

\begin{Proposition}\label{delaroe}
There is an  isomorphism $\Gamma:{\rm C}^*_{\rm red}\big[\Xi(X_0,\delta)\big]\to{\rm Roe}(X_0,\delta)$\,, with an isomorphic restriction $\Gamma:C_{\rm c}\big[\Xi(X_0,\delta)\big]\to{\rm Band}(X_0,\delta)$\,.  This isomorphism sends the ideal ${\rm C}^*(X_0\!\times\!X_0)$\,, corresponding to the main orbit, onto the ideal of all compact operators in $\ell^2(X_0)$\,.
\end{Proposition}

Let us recall the action of the isomorphism on continuous compactly supported functions. Such a function $h$ is actually supported on some set $\overline\Delta_r$\,, so it may be seen as a continuous function $:\Delta_r\!\to\mathbb C$\,. Then $\Gamma$ transforms it into the operator \eqref{fraier} with $H=h$\,. 

\begin{Definition}\label{propa}
Our metric space $(X_0,\delta)$ possesses Yu's {\rm Property A} if there is a sequence $\{K_n:X_0\!\times\!X_0\to\mathbb R\!\mid\!n\in\N\}$ of positive definite band kernels which converge to the constant function $1$ uniformly on each set $\Delta_r$\,.
\end{Definition}

This is a  weak amenability condition, satisfied by large classes of metric spaces (it is quite difficult to provide counter-examples), admitting various reformulations and playing an important role in several parts of mathematics \cite{Tu,Will}. For us it is important since, by \cite[Th.\,5.3]{STY}, {\it it is equivalent with the amenability of the groupoid $\Xi(X_0,\delta)$} (and, as a fact, with the nuclearity of the uniform Roe algebra). As a conclusion, we have

\begin{Proposition}\label{ocluzie}
If the countable metric space $(X_0,\delta)$ is uniformly discrete, it has bounded geometry and the Property A, then the groupoid $\Xi(X_0,\delta)$ is manageable and standard, with main orbit $X_0$\,.
\end{Proposition}

\smallskip
Let us define $\mathbf S_0$ to be the family of subsets $M_0\subset X_0$ with finite complement. For every $M_0\in\mathbf S_0$ we set $M\!:=M_0\cup X_\infty\subset X$ and 
\begin{equation*}\label{impera}
\T_0:=\{M\!\mid\!M_0\in\mathbf S_0\}\,,\quad\T:=\mathbf T_0\cup\{X_\infty\}\,.
\end{equation*} 
Clearly $\T_0$ is the family of subsets of the Stone-\u Cech compactification which have a finite complement contained in $X_0$\,. Its elements are all clopen sets. We need to understand $\T$ in terms of the Fell topology.

\begin{Lemma}\label{pisicher}
The family $\T$ is closed in $\Cl(X)$ and $\!\underset{M\in\T_0}{\lim}\!M\!=X_\infty$\,.
\end{Lemma}

\begin{proof}
Note that $\T_0$ is a net in $\Cl(X)$\,, with the order given by inverse inclusion: $M\le N$ if and only if $M\supset N$. Let us show that $M\to X_\infty$\,, by checking the point (a) of Lemma \ref{indiciaza} (the point (b) clearly holds since $M\supset X_\infty$). Obviously, this implies that $\T$ is closed.

\smallskip
If $M\ni x_M\to x\in X$, we need to show that $x\in X_\infty$\,. If not then $x\in X_0$\,, and since the singleton $\{x\}$ is open in $X_0$ then $x_M=x$\, eventually. Hence there exists $N_0\subset X_0$ with finite complement such that $x=x_M\in M_0$ for every $M_0\subset N_0$ with finite complement. This is clearly absurd (set $M_0=N_0\!\setminus\!\{x\}$).
\end{proof}

\begin{Remark}\label{subceva}
{\rm It is clear that $\T_0$ can be replaced by one of its subnets. If, for example, $x_0\in X_0$ and $\{R_n\!\mid\! n\in \N\}$ is a sequence of positive numbers with $R_n\to\infty$\,, then
$$
\T_0':=\big\{B_{x_0}(R_n)^c\,\big\vert\, n\in\N\big\}
$$
is such a subnet and one may state a result analog to Theorem \ref{uvert}.}
\end{Remark}

Note that all the ingredients of Setting \ref{setter} are present. The mapping $\th:=(\r,\d)^{-1}$ here is just the identity, so \eqref{vinalla} now reduces to \eqref{fraier}. {\it Then Theorem \ref{uvert} straightforwardly follows from Theorem \ref{consecinte}.} Remark \ref{vine} also applies.

\begin{Remark}\label{grafuri}
{\rm One may give to operators of the form \eqref{fraier} a graph-theoretical flavor, making the connection with \cite{LS}. Suppose that $H$ is a symmetric matrix with positive entries, i.\,e.\;$H(x,y)=H(y,x)\ge 0$ for every $x,y\in X_0$\,. If in fact $H(x,y)$ is a positive integer, one may think that $X_0$ is the vertex set of an undirected graph and that there are precisely $H(x,y)$ edges connecting $x$ to $y$\,. The connectivity in the graph is given simply by $x\sim y$ if and only if $H(x,y)\ne 0$\,. Then the expression $\sup_x\sum_y\!H(x,y)=:\sup_x{\rm deg}(x)$\,, appearing in the Schur estimate and connected to the Hahn norm as in Remark \ref{vine}, may be seen as as the upper bound over $x$  of the number of neighbors of the vertex $x$ "multiplicities included". Operators as \eqref{fraier} are then a sort of adjacency matrix of the graph, and the Laplace operator version
$$
(\Delta u)(x):=\sum_{y\in X_0}\!H(x,y)\big(u(x)-u(y)\big)
$$
may also be dealt with. If $H$ is no longer integer-valued, we can still think of such operators as weighted adjacency or weighted Laplacian on graphs. Unbounded versions are available, very interesting, and not yet treated by our groupoid methods. For the full formalism and many applications, we refer to \cite{FLW}. Note that passing from ${\sf H}$ to ${\sf H}_{M_0}$ boils down to erase both the vertices belonging to $X_0\!\setminus\!M_0$ and the edges with one end in such vertices. Of course, this corresponds to the notion of a vertex-induced subgraph.}
\end{Remark}

\begin{Example}\label{crupuri}
{\rm We indicate very briefly the particular {\it group case.} If $X_0$ is a countable discrete group, it is known \cite[2.3]{Will} that there exist left-invariant metrics $\delta$ such that $(X_0,\delta)$ is a uniformly discrete metric space with bounded geometry (word metrics in the finitely generated case are particular examples). Property A is very common among these metric spaces. It is much more general than amenability (all the free groups on a finite number of generators have property A). The property is stable under subgroups, extensions, wreath products, amalgamated free products and inductive limits with injective connection maps. Many other examples may be found in \cite[2.3]{Will}. In the group case there is a third point of view upon the $C^*$-algebra ${\rm C}^*\big[\Xi(X_0,\delta)\big]\cong{\rm Roe}(X_0,\delta)$\,. It is isomorphic with the crossed product $\ell^\infty(X_0)\!\rtimes\!X_0$\,, coming from the left action of $X_0$ on itself, that extends to a continuous action on the Stone-\u Cech compactification $\beta X_0$ seen as the Gelfand spectrum of the Abelian $C^*$-algebra $\ell^\infty(X_0)$\,. Then the non-invariant restrictions to the sets $M\subset\beta X_0$ correspond to crossed products with partial actions \cite{Ab,Ex1,Ex2}. However the groupoid approach, besides being more general, seems the best $C^*$-algebraical strategy to prove Persson-type formulae.
}
\end{Example}

{\bf Acknowledgements:} The author has been supported by the Fondecyt Project 1200884. He is grateful to an anonymous referee for a careful reading of the manuscript and for valuable remarks. A useful discussion with Professor Jiawen Zhang  is also acknowledged.


 \end{document}